\documentclass[12pt]{article} 
\usepackage[sectionbib]{natbib}
\usepackage{array,epsfig,fancyheadings,rotating}
\usepackage{longtable,booktabs}
\usepackage[]{hyperref}  
\usepackage{sectsty, secdot}
\sectionfont{\fontsize{12}{14pt plus.8pt minus .6pt}\selectfont}
\renewcommand{\theequation}{\thesection\arabic{equation}}
\subsectionfont{\fontsize{12}{14pt plus.8pt minus .6pt}\selectfont}

\usepackage{amsmath}
\usepackage{amssymb}
\usepackage{amsfonts}
\usepackage{multirow}
\usepackage{amsthm}
\usepackage{float}

\newtheorem{lemma}{Lemma}
\newtheorem{corollary}{Corollary}
\newtheorem{proposition}{Proposition}
\theoremstyle{definition}

\usepackage{xcolor}

\newcommand{\bs}[1]{\boldsymbol{#1}}
\newcommand{\bo}{\mathbf}
\newtheorem{assumption}{Assumption}
\usepackage{colortbl}

\begin{document}


\renewcommand{\baselinestretch}{2}


\markboth{\hfill{\footnotesize\rm JONI VIRTA AND KLAUS NORDHAUSEN} \hfill}
{\hfill {\footnotesize\rm} \hfill}

\renewcommand{\thefootnote}{}
$\ $\par


\fontsize{12}{14pt plus.8pt minus .6pt}\selectfont \vspace{0.8pc}
\centerline{\large\bf DETERMINING THE SIGNAL DIMENSION IN}
\vspace{2pt} \centerline{\large\bf SECOND ORDER SOURCE SEPARATION}
\vspace{.4cm} \centerline{Joni Virta$^1$ and Klaus Nordhausen$^2$} \vspace{.4cm} \centerline{\it
$^1$Aalto University, Finland}
\centerline{\it $^2$Vienna University of Technology, Austria} \vspace{.55cm} \fontsize{9}{11.5pt plus.8pt minus
.6pt}\selectfont


\begin{quotation}
\noindent {\it Abstract:}
While an important topic in practice, the estimation of the number of non-noise components in blind source separation has received little attention in the literature. Recently, two bootstrap-based techniques for estimating the dimension were proposed, and although very efficient, they suffer from the long computation times caused by the resampling. We approach the problem from a large sample viewpoint and develop an asymptotic test for the true dimension. Our test statistic based on second-order temporal information has a very simple limiting distribution under the null hypothesis and requires no parameters to estimate. Comparisons to the resampling-based estimates show that the asymptotic test provides comparable error rates with significantly faster computation time. An application to sound recording data is used to illustrate the method in practice.

\vspace{9pt}
\par
\end{quotation}\par

\def\thefigure{\arabic{figure}}
\def\thetable{\arabic{table}}

\renewcommand{\theequation}{\thesection.\arabic{equation}}

\fontsize{12}{14pt plus.8pt minus .6pt}\selectfont

\setcounter{equation}{0} 

\section{Introduction}


The modelling of multivariate time series is notoriously difficult and an increasingly common option is to use latent variable or factor models \citep[see for example][and the references therein]{Ensor2013,ChangGuoYao2018}. In this paper we will follow the \textit{blind source separation} (BSS) approach, as an intermediary step prior to modelling. In BSS the observed multivariate time series is bijectively decomposed into several univariate time series that exhibit some form of mutual independence, such as second order uncorrelatedness or even full statistical independence. After such a decomposition, the lack of interaction between the univariate series allows us to model them separately, requiring much smaller numbers of parameters.

A particularly popular choice among BSS models for time series is the second order source separation (SOS) model \citep{ComonJutten2010} which assumes that the observed zero-mean $ p $-variate time series $ \textbf{x}_1, \ldots , \textbf{x}_T $ is generated as,
\begin{align}\label{eq:model}
\textbf{x}_t = \boldsymbol{\Omega} \textbf{z}_t, \quad t = 1, \ldots , T,
\end{align}
where the \textit{source} series $ \textbf{z}_1, \ldots , \textbf{z}_T$ is a latent non-degenerate, zero-mean, second-order stationary $ p $-variate time series with uncorrelated component series and $ \boldsymbol{\Omega} \in \mathbb{R}^{p \times p} $ is an unknown, invertible matrix-valued parameter. The assumption of zero mean is without loss of generality as we can always center our series. The objective in model \eqref{eq:model} is to estimate an inverse $ \hat{\boldsymbol{\Gamma}} $ for $ \boldsymbol{\Omega} $, giving us an estimate $ \hat{\textbf{z}}_t = \hat{\boldsymbol{\Gamma}} \hat{\textbf{x}}_t$ for the $ p $ sources, which can then further be modelled univariately. 

However, noise is often an inevitable part of any real world signal and we incorporate it in the model \eqref{eq:model} through the sources. That is, we assume that the sources consist of two parts, $ \textbf{z}_t = (\textbf{z}_{1t}^\top, \textbf{z}_{2t}^\top)^\top $, where $ \textbf{z}_{1t} \in \mathbb{R}^{d} $ contains the signals and $ \textbf{z}_{2t} \in \mathbb{R}^{p - d} $ is a white noise vector. To avoid overfitting in the modelling phase, a crucial step in BSS is to identify the noise subvector $ \textbf{z}_{2t} $ among the sources and discard it prior to the modelling. This problem, signal dimension estimation, has only recently been considered in the context of statistical blind source separation and in this paper we propose a novel estimate that relies on asymptotic hypothesis testing. But first, we review two classical SOS-methods that serve as the basis for both our method and the two existing ones.


The standard way of estimating the sources in \eqref{eq:model} is via second-order temporal moments. In algorithm for multiple signals extraction (AMUSE) \citep{TongSoonHuangLiu1990}, an estimate $ \hat{\boldsymbol{\Gamma}} $ is obtained from the generalized eigendecomposition,
\begin{equation*}
\hat{\boldsymbol{\Gamma}} \hat{\textbf{S}}_0 \hat{\boldsymbol{\Gamma}}^\top = \textbf{I}_p \quad \mbox{and} \quad \hat{\boldsymbol{\Gamma}} \hat{\textbf{R}}_\tau \hat{\boldsymbol{\Gamma}}^\top = \hat{\textbf{D}}_\tau,
\end{equation*}
where $ \hat{\textbf{S}}_0 = (1/T) \sum_{t=1}^T (\textbf{x}_t - \bar{\textbf{x}}) (\textbf{x}_t - \bar{\textbf{x}})^\top$ is the marginal covariance matrix, $ \hat{\textbf{R}}_{\tau} = (\hat{\textbf{S}}_{\tau} + \hat{\textbf{S}}{}_{\tau}^\top )/2$, $ \hat{\textbf{S}}_{\tau} = [1/(T - \tau)] \sum_{t=1}^{T - \tau} (\textbf{x}_t  - \bar{\textbf{x}}) (\textbf{x}_{t + \tau}  - \bar{\textbf{x}})^\top$ is the $ \tau $-lag autocovariance matrix and $ \hat{\textbf{D}}_\tau $ is a diagonal matrix. If the lag-$ \tau $ autocovariances of the latent series are distinct, then $ \hat{\boldsymbol{\Gamma}} $ is consistent for $ \boldsymbol{\Omega}^{-1} $ up to permutation and signs of its rows. For the statistical properties of AMUSE, see \citet{MiettinenNordhausenOjaTaskinen2012}.

A usually more efficient estimate, which does not depend so much on the selection of a single parameter $\tau$, is given by second order blind identification (SOBI) \citep{BelouchraniAbedMeraimCardosoMoulines1997}, an extension of AMUSE to multiple autocovariance matrices. In SOBI we choose a set of lags $ \mathcal{T} $ and estimate the orthogonal matrix $ \hat{\textbf{U}} $ by maximizing,
\begin{align}\label{eq:sobi_diagonalization} 
\sum_{\tau \in \mathcal{T}} \left\| \mathrm{diag} \left( \textbf{U}^\top \hat{\textbf{S}}_{0}^{-1/2} \hat{\textbf{R}}_{\tau} \hat{\textbf{S}}_{0}^{-1/2} \textbf{U}  \right) \right\|^2,
\end{align}
over the set of all orthogonal $ \textbf{U} \in \mathbb{R}^{p \times p} $, where $ \hat{\textbf{S}}_{0}^{-1/2} $ is the unique symmetric inverse square root of the almost surely positive definite matrix $ \hat{\textbf{S}}_{0} $. This procedure is called orthogonal (approximate) joint diagonalization and provides a natural extension of the generalized eigendecomposition to more than two matrices. Note that this makes AMUSE a special case of SOBI with $ | \mathcal{T} | = 1 $. An estimate for $ \boldsymbol{\Omega}^{-1} $ is given by $ \hat{\boldsymbol{\Gamma}} = \hat{\textbf{U}}{}^\top \hat{\textbf{S}}{}_{0}^{-1/2} $ and it is consistent (up to permutation and signs of its rows) if, for all pairs of sources, there exists a lag $ \tau \in \mathcal{T} $, such that the lag-$ \tau $ autocovariances of the two sources differ, see \citet{BelouchraniAbedMeraimCardosoMoulines1997,MiettinenNordhausenOjaTaskinen2014b,MiettinenIllnerNordhausenOjaTaskinenTheis2016}. For more details about solving the maximization problem in (\ref{eq:sobi_diagonalization}) see for example \citet{IllnerMiettinenFuchsTaskinenNordhausenOjaTheis2015} and the references therein.

We now turn back to our problem at hand, the signal dimension estimation. Of the two estimates proposed in literature, both rely on SOBI (with AMUSE as a special case) and the first \citep{MatilainenNordhausenVirta2017} bases the approach on the following set of hypotheses for $ q = 0, \ldots , p - 1 $, 
\begin{align}\label{eq:null_hypotheses}
H_{0q}: \mbox{ $\textbf{z}_t$ contains a $ p - q $-subvector of white noise.}
\end{align}
A suitable test statistic for $ H_{0q} $ is given, e.g., by the mean of the last $ p - q $ squared diagonal elements of $ \hat{\textbf{U}}{}^\top \hat{\textbf{S}}{}_{0}^{-1/2} \hat{\textbf{R}}_{\tau} \hat{\textbf{S}}{}_{0}^{-1/2} \hat{\textbf{U}}  $ over all $ \tau \in \mathcal{T} $, where $ \hat{\textbf{U}} $ is the maximizer of \eqref{eq:sobi_diagonalization}. This is based on the fact that all autocovariances of white noise series vanish, see Section \ref{sec:main} for a more detailed motivation. \cite{MatilainenNordhausenVirta2017} use bootstrapping to obtain the null distributions of the test statistics and sequence several of the tests together to estimate the signal dimension $ d $, see the end of Section \ref{sec:main} for various strategies. Similar techniques have been used for the dimension estimation of iid data in \cite{NordhausenOjaTyler2016, NordhausenOjaTylerVirta2017}.  

An alternative approach is proposed by \cite{NordhausenVirta2018} who extend the \textit{ladle estimate} of \cite{LuoLi2016} to the time series framework. The estimate is based on combining the classical scree plot with the bootstrap variability \citep{YeWeiss2003} of the joint diagonalizer and has the advantage of estimating the dimension directly, without any need for hypothesis testing.

We complement these approaches by devising an asymptotic test for the null hypotheses \eqref{eq:null_hypotheses}, operating under semiparametric assumptions on the source distributions. Using simulations, the test is showed to enjoy the standard properties of asymptotic tests, computational speed and efficiency under time series of moderate and large lengths. The first of these properties is especially important and desirable, considering that the only competitors of the proposed method are based on computationally costly data resampling techniques. Moreover, the mathematical form of the proposed asymptotic test is shown to be particularly simple and elegant.

The paper is structured as follows. In Section \ref{sec:main} we present our main results and discuss the implications and strictness of the assumptions required for them to hold. Section \ref{sec:theory} contains the technical derivations that lead to the results in Section \ref{sec:main} and can be safely skipped by a casual reader. The proofs of the results are collected in Appendix \ref{sec:appendix}. Section \ref{sec:simulations} sees us comparing the proposed dimension estimate to the bootstrap- and ladle estimates under various settings using simulated data. In Section \ref{sec:data_example} we apply the proposed method to estimate the dimension of a sound recording data set and in Section \ref{sec:discussion}, we finally conclude with some prospective ideas for future research.

\section{Main results}\label{sec:main}

In this section we present our main results and the assumptions required by them. The more technical content is postponed to Section \ref{sec:theory} and can be skipped if the reader is not interested in the theory behind the results.


Let the observed time series $ \textbf{x}_t $ come from the SOS-model \eqref{eq:model} and denote by $ \lambda^*_{\tau k} $ the $ \tau $-lag autocovariance of the $ k $th component of $ \textbf{z}_t $. Recall that SOBI jointly diagonalizes the set of standardized and symmetrized autocovariance matrices $ \hat{\textbf{H}}_\tau = \hat{\textbf{S}}{}_{0}^{-1/2} \hat{\textbf{R}}_{\tau} \hat{\textbf{S}}{}_{0}^{-1/2} $, $ \tau \in \mathcal{T} $, to obtain the orthogonal matrix $ \hat{\textbf{U}} $. Order the columns of $ \hat{\textbf{U}} $ such that the sums of squared pseudo-eigenvalues, $ \sum_{\tau \in \mathcal{T}} \mathrm{diag}(\hat{\textbf{U}}{}^\top \hat{\textbf{H}}_\tau \hat{\textbf{U}})^2$, are in a decreasing order and partition  $ \hat{\textbf{U}} $ as $ ( \hat{\textbf{V}}_q, \hat{\textbf{W}}_q) $ where $ \hat{\textbf{V}}_q \in \mathbb{R}^{p \times q}, \hat{\textbf{W}}_q \in \mathbb{R}^{p \times (p - q) } $ for some fixed $ q $. We show in Section \ref{sec:theory} that, for large $ T $, this ordering places the noise components after the signals in the estimated sources. If the null hypothesis $ H_{0q} $ is true, the autocovariance matrices of the last $ p - q $ estimated sources,
\[ 
\hat{\textbf{D}}_{\tau q} = \hat{\textbf{W}}_q^\top \hat{\textbf{H}}_\tau \hat{\textbf{W}}_q,
\]
are then expected to be close to zero matrices due to the last sources being (at least, asymptotically) white noise. To accumulate information over multiple lags, we use as our test statistic for $ H_{0q} $ the mean of the squared elements of the matrices $ \hat{\textbf{D}}_{\tau q} $ over a fixed set of lags $ \tau \in \mathcal{T} $,
\[ 
\hat{m}_q = \frac{1}{| \mathcal{T}| (p - q)^2} \sum_{\tau \in \mathcal{T}} \| \hat{\textbf{D}}_{\tau q} \|^2,
\] 
which likewise measures departure from the null hypothesis $ H_{0q} $. In the special case of AMUSE we have only a single matrix $ \hat{\textbf{D}}_{\tau q} $, which can in practice be obtained using the easier-to-compute generalized eigendecomposition, instead of joint diagonalization. Note that it is possible for the matrices $ \hat{\textbf{D}}_{\tau q} $ to be small in magnitude even if the number of white noise sources in the model is less than $ p - q $. This situation can arise if some of the signal series are indistinguishable from white noise based on autocovariances alone and as such we need to restrict the set of signal distributions we can consider. The next assumption guarantees that each signal component exhibits non-zero autocovariance for at least one lag $ \tau \in \mathcal{T} $, and is thus distinguishable from white noise.

\begin{assumption}\label{assu:signal_from_noise}
	For all $ k = 1, \ldots , d $, there exists $ \tau \in \mathcal{T} $ such that $ \lambda^*_{\tau k} \neq 0$.
\end{assumption}

Considering that most signals encountered in practice exhibit autocorrelation, Assumption \ref{assu:signal_from_noise} is rather non-restrictive. Moreover, we can always increase the number of feasible signal processes by incorporating more lags in $ \mathcal{T} $. However, there exists time series which, while not being white noise, still have zero autocorrelation for all finite lags. For example, stochastic volatility models \citep[see for example][]{MikoschEtAl2009} belong to this class of processes, and consequently, by Assumption \ref{assu:signal_from_noise}, they are excluded from our model (see, however Section \ref{sec:discussion} for an idea on how to incorporate these distributions in the model).

The second assumption we need is more technical in nature and requires that the source series come from a specific, wide class of stochastic processes. A similar assumption is utilized also in \citet{MiettinenNordhausenOjaTaskinen2012,MiettinenNordhausenOjaTaskinen2014b,MiettinenIllnerNordhausenOjaTaskinenTheis2016}.
\begin{assumption}\label{assu:ma_infinity}
	The latent series $ \textbf{z}_t $ are linear processes having the $ \mathrm{MA}(\infty) $-representation,
	\[ 
	\textbf{z}_t = \sum_{j=-\infty}^\infty \boldsymbol{\Psi}_j \boldsymbol{\epsilon}_{t - j},
	\]
	where $ \boldsymbol{\epsilon}_t \in \mathbb{R}^p$ are second-order standardized, iid random vectors with exchangeable, marginally symmetric components having finite fourth order moments and $ \boldsymbol{\Psi}_j \in \mathbb{R}^{p \times p} $ are diagonal matrices satisfying $ \sum_{j=-\infty}^\infty \boldsymbol{\Psi}_j^2 = \textbf{I}_p $ and $ \| \sum_{j=-\infty}^\infty |\boldsymbol{\Psi}_j| \| < \infty$ where $ | \boldsymbol{\Psi}_j | \in \mathbb{R}^{p \times p}$ denotes the matrix of component-wise absolute values of $ \boldsymbol{\Psi}_j $. Moreover the lower right $ (p - d) \times (p - d) $ blocks of $ \boldsymbol{\Psi}_j $ (the noise) equal $ \boldsymbol{\Psi}_{j00} = \delta_{j0} \textbf{I}_{p - d} $, where $ \delta_{\cdot \cdot} $ is the Kronecker delta.
\end{assumption}

Note that all second-order stationary multivariate time series can by Wold's decomposition be given a $ \mathrm{MA}(\infty) $-representation, meaning that the most stringent part of Assumption \ref{assu:ma_infinity} is that it requires the innovations of the sources to have identical, symmetric marginal distributions. The importance of Assumption \ref{assu:ma_infinity} to the theory comes from the fact that under it the joint limiting distribution of the sample autocovariance matrices can be derived. As such, it could also be replaced with some other assumption guaranteeing the same thing.

With the previous, we are now able to present our main result.
\begin{proposition}\label{prop:1}
	Under Assumptions \ref{assu:signal_from_noise}, \ref{assu:ma_infinity} and the null hypothesis $ H_{0q} $,
	\[ 
	T |\mathcal{T}| (p - q)^2 \cdot \hat{m}_q \rightsquigarrow \chi^2_{| \mathcal{T} |(p - q)(p - q + 1)/2},
	\]
	where $ \chi^2_\nu $ denotes the chi-squared distribution with $ \nu $ degrees of freedom.
\end{proposition}

The limiting distribution in Proposition \ref{prop:1} is remarkably simple, does not depend on the type of white noise and requires no parameters to estimate, implying that it is also fast to use in practice. Note that the number of degrees of freedom of the limiting distribution is equal to the total number of free elements in the symmetric matrices $ \hat{\textbf{D}}_{\tau q} $, $ \tau \in \mathcal{T} $. Thus, each of the elements can be seen to asymptotically contribute a single $ \chi^2_1 $ random variable to the test statistic.

To use Proposition \ref{prop:1} to estimate the signal dimension in the $ p $-dimensional BSS-model \eqref{eq:model}, we sequence together a set of asymptotic tests for the null hypotheses $ H_{00}, H_{01}, \ldots , H_{0(p - 1)} $. Denote the string of $ p $-values produced by these tests by $ (p_0, p_1, \ldots , p_{p - 1}) $ and fix a level of significance $ \alpha $. Different estimates for $ d $ are now obtained by considering the $ p $-values via various strategies. The forward estimate of $ d $ is the smallest $ q $ for which $ p_q \geq \alpha $. The backward estimate of $ d $ is $ q + 1 $ where $ q $ is the largest value for which $ p_q < \alpha $. The divide-and-conquer estimate is obtained by iteratively halving the search interval until a change point from $ < \alpha $ to $ \geq \alpha $ is found.

\section{Theoretical derivations}\label{sec:theory}

Throughout this section, we assume the SOS-model \eqref{eq:model} and a fixed set of lags $ \mathcal{T} = \{ \tau_1, \ldots , \tau_{|\mathcal{T}|} \}$. Moreover, we work under the assumption of identity mixing, $ \boldsymbol{\Omega} = \textbf{I}_p $, which is without loss of generality as SOBI is affine equivariant, meaning that the source estimates do not depend on the value of $ \boldsymbol{\Omega} $ \citep{MiettinenIllnerNordhausenOjaTaskinenTheis2016}. To ensure identifiability of $ \boldsymbol{\Omega} $ we may further set $ \textbf{S}_0 = \mathrm{E}(\textbf{x}_t \textbf{x}_t^\top) = \textbf{I}_p $. We assume a fixed null hypothesis $ H_{0q} $ and denote the number of white noise components by $ r = p - q $.

The population autocovariance matrices are denoted by $ \textbf{S}_\tau = \mathrm{E}(\textbf{x}_t \textbf{x}_{t + \tau}^\top) $ and $ \textbf{R}_\tau = ( \textbf{S}_\tau + \textbf{S}_\tau^\top )/2$, and by the identity mixing and uncorrelatedness of the latent series we have,
\[ 
\textbf{S}_\tau = \textbf{R}_\tau = \textbf{D}_\tau = \begin{pmatrix}
\boldsymbol{\Lambda}_\tau & \textbf{0} \\
\textbf{0} & \textbf{0},
\end{pmatrix},
\]
where $ \boldsymbol{\Lambda}_\tau $ is a $ q \times q $ diagonal matrix, $ \tau \in \mathcal{T} $. The lower right block of the matrix $ \textbf{D}_\tau $ vanishes for all $ \tau \in \mathcal{T} $ as autocovariances of a white noise series are zero. Without loss of generality, we assume that the signals are ordered in $ \textbf{z}_{1t} $ such that the diagonal elements of $ \sum_{\tau \in \mathcal{T}} \boldsymbol{\Lambda}^2_\tau $ are in decreasing order. Moreover, if there are ties, we fix the order by ordering the tied components in decreasing order with respect to the diagonal elements of $ \boldsymbol{\Lambda}_{\tau_1}^2 $. If we still have ties, we order the tied components in decreasing order with respect to the diagonal elements of $ \boldsymbol{\Lambda}_{\tau_2}^2 $ and so on. If after all this we still have tied components, we set them in arbitrary order and note that such sets of tied components have the same autocovariance structure for all lags $ \tau \in \mathcal{T} $, making them indistinguishable by SOBI. However, this just makes the individual signals unestimable and does not affect the estimation of the dimension in any way, as long as Assumption \ref{assu:signal_from_noise} holds.

Partition then the signals into $ v $ groups such that each group consists solely of signals with matching autocovariance structures on all lags $ \tau \in \mathcal{T} $ and such that each pair of distinct groups has a differing autocovariance for at least one lag $ \tau \in \mathcal{T} $. The autocovariance of the $ j $th group for lag $ \tau $ is denoted by $ \lambda_{\tau j} $ and the size of the $ j $th group by $ p_j $, implying that $ p_1 + \cdots p_v = q$. By Assumption \ref{assu:signal_from_noise}, the white noise forms its own group not intersecting with any of the signal groups, and in the following we refer to the noise group with the index $ 0 $, as in $ p_0 = r $ and $ \lambda_{\tau 0} = 0 $ for all $ \tau \in \mathcal{T} $. If $ v = 1 $, all signal components are indistinguishable by their autocovariances and in the other extreme, $ v = q $, no ties occurred in ordering the signals and each signal pair has differing autocovariances for at least one lag $ \tau \in \mathcal{T} $.

We introduce yet one more assumption which is actually implied by Assumption \ref{assu:ma_infinity} and is as such not strictly necessary. However, some of the following auxiliary results are interesting on their own and can be shown to hold under Assumption \ref{assu:consistent}, without the need for Assumption \ref{assu:ma_infinity}.

\begin{assumption}\label{assu:consistent}
	The sample covariance matrix and the sample autocovariance matrices are root-$ T $ consistent, $ \sqrt{T} ( \hat{\textbf{S}}_\tau - \textbf{D}_\tau) = \mathcal{O}_p(1)$, for $ \tau \in \mathcal{T} \cup \{ 0 \} $, where $ \textbf{D}_0 = \textbf{I}_p $.  
\end{assumption}

We begin with a simple linearization result for the standardized autocovariance matrices.  The notation $ \hat{\textbf{H}}_{\tau00}, \hat{\textbf{R}}_{\tau00} $ in Lemma \ref{lem:0} refers to the lower right $ r \times r $ diagonal blocks of the matrices $ \hat{\textbf{H}}_{\tau} = \hat{\textbf{S}}{}_{0}^{-1/2} \hat{\textbf{R}}_{\tau} \hat{\textbf{S}}{}_{0}^{-1/2} $ and $ \hat{\textbf{R}}_{\tau} $. Under $ H_{0q} $ these sub-matrices gather the autocovariances of the noise components.

\begin{lemma}\label{lem:0}
	Under Assumption \ref{assu:consistent} we have
	\[ 
	\hat{\textbf{H}}_\tau = \hat{\textbf{R}}_\tau + \mathcal{O}_p(1/\sqrt{T}), \quad \mbox{for all } \tau \in \mathcal{T}.
	\]
	If $ H_{0q} $ further holds, then,
	\[ 
	\hat{\textbf{H}}_{\tau00} = \hat{\textbf{R}}_{\tau00} + \mathcal{O}_p(1/T), \quad \mbox{for all } \tau \in \mathcal{T}.
	\]
\end{lemma}

Our second auxiliary result shows that, under Assumptions  \ref{assu:signal_from_noise} and \ref{assu:consistent}, the SOBI solution is, while not identifiable, of a very specific asymptotic form (up to permutation). The block division and indexing in Lemma \ref{lem:1} are based on the division of the sources into the $ v + 1 $ groups of equal autocovariances.

\begin{lemma}\label{lem:1}
	Under Assumptions \ref{assu:signal_from_noise},  \ref{assu:consistent} and the null hypothesis $ H_{0q} $, there exists a sequence of permutation matrices $ \hat{\textbf{P}} $ such that,
	\[
	\hat{\textbf{U}} \hat{\textbf{P}} =
	\left[\begin{array}{cccc}
	\cellcolor{black!10}\hat{\textbf{U}}_{11} & \cdots & \hat{\textbf{U}}_{1v} & \hat{\textbf{U}}_{10} \\
	\vdots & \cellcolor{black!10}\ddots & \vdots & \vdots\\
	\hat{\textbf{U}}_{v1} & \cdots & \cellcolor{black!10}\hat{\textbf{U}}_{vv} & \hat{\textbf{U}}_{v0} \\
	\hat{\textbf{U}}_{01} &\cdots & \hat{\textbf{U}}_{0v} & \cellcolor{black!10}\hat{\textbf{U}}_{00}
	\end{array}\right] ,
	\]
	where the diagonal blocks (shaded) satisfy $ \hat{\textbf{U}}_{ii} = \mathcal{O}_p(1)$ and the off-diagonal blocks satisfy $ \hat{\textbf{U}}_{ij} = \mathcal{O}_p(1/\sqrt{T})$.
\end{lemma}

\begin{corollary}\label{cor:1}
	Under the assumptions of Lemma \ref{lem:1}, we have for each $ j = 0, 1, \ldots , v $ that,
	\[ 
	\hat{\textbf{U}}_{jj}^\top \hat{\textbf{U}}_{jj} - \textbf{I}_{p_j} = \mathcal{O}_p(1/T) \quad \mbox{and} \quad \hat{\textbf{U}}_{jj} \hat{\textbf{U}}_{jj}^\top - \textbf{I}_{p_j} = \mathcal{O}_p(1/T).
	\]
\end{corollary}

The first $ v $ diagonal blocks in the block matrix of Lemma \ref{lem:1} correspond to the groups of signals that are mutually indistinguishable and the final diagonal block to the $ r $ noise components (which are also indistinguishable from each other). The main implication of Lemma \ref{lem:1} is that the sources within a single group can not be separated by SOBI but the signals coming from two different groups can be, the mixing vanishing at the rate of root-$ T $. In the special case of $ p_j = 1 $, for all $ j = 0, 1, \ldots , v $, Lemma \ref{lem:1} is an instant consequence of \cite[Theorem 1(ii)]{MiettinenIllnerNordhausenOjaTaskinenTheis2016}.  

The next lemma states that our test statistic is under the null asymptotically equivalent to a much simpler quantity, not depending on the estimation of the SOBI-solution $ \hat{\textbf{U}} $.


\begin{lemma}\label{lem:2}
	Under Assumptions \ref{assu:signal_from_noise}, \ref{assu:consistent} and the null hypothesis $ H_{0q} $, we have,
	\[ 
	T \cdot \hat{m}_q = T \cdot \hat{m}^*_q + o_p(1),
	\]
	where
	\[ 
	\hat{m}^*_q = \frac{1}{ |\mathcal{T}| r^2 } \sum_{\tau \in \mathcal{T}} \| \hat{\textbf{R}}_{\tau00} \|^2,
	\]
	and $ \hat{\textbf{R}}_{\tau00} $ is the lower right $ r \times r $ block of $ \hat{\textbf{R}}_{\tau} $.
\end{lemma}

To compute the limiting distribution of the proxy $ \hat{m}_q^* $ we next show that the joint limiting distribution of the blocks $ \hat{\textbf{R}}_{\tau00} $ is under Assumption~\ref{assu:ma_infinity} and $ H_{0q} $ conveniently a multivariate normal distribution. The result is a slight modification of \cite[Lemma 1]{MiettinenIllnerNordhausenOjaTaskinenTheis2016}. In the statement of Lemma \ref{lem:3}, $ \textbf{J}_r $ denotes the $ r \times r $ matrix filled with ones and $ \textbf{E}_r^{i j} $ denotes the $ r \times r $ matrix filled otherwise with zeroes but with a single one as the $ (i, j) $th element.

\begin{lemma}\label{lem:3}
	Under Assumption \ref{assu:ma_infinity} and the null hypothesis $ H_{0q} $, the blocks $  \hat{\textbf{R}}_{\tau_100}, \ldots , \hat{\textbf{R}}_{\tau_{|\mathcal{T}|}00} $ have a joint limiting normal distribution,
	\[ 
	\sqrt{T} \mathrm{vec} \left( \hat{\textbf{R}}_{\tau_100}, \ldots , \hat{\textbf{R}}_{\tau_{|\mathcal{T}|}00} \right) \rightsquigarrow \mathcal{N}_{|\mathcal{T}| r^2} (\textbf{0}, \textbf{V}),
	\]
	where $ \mathrm{vec} $ is the column-vectorization operator,
	\[ 
	\textbf{V} = \begin{pmatrix}
	\textbf{V}_0 & \textbf{0} & \cdots & \textbf{0} \\
	\textbf{0} & \textbf{V}_0 & \cdots & \textbf{0} \\
	\vdots & \vdots & \ddots & \vdots \\
	\textbf{0} & \textbf{0} & \cdots & \textbf{V}_0
	\end{pmatrix} \in \mathbb{R}^{|\mathcal{T}| r^2 \times |\mathcal{T}| r^2},
	\]
	and $ \textbf{V}_0 = \mathrm{diag}(\mathrm{vec}(\textbf{J}_{r} + \textbf{I}_{r})/2) (\textbf{K}_{rr} - \textbf{D}_{rr} + \textbf{I}_{r^2} )$ where $ \textbf{K}_{rr} = \sum_{i=1}^r \sum_{j=1}^r \textbf{E}_r^{ij} \otimes \textbf{E}_r^{ji}$ and $ \textbf{D}_{rr} = \sum_{i=1}^r \textbf{E}_r^{ii} \otimes \textbf{E}_r^{ii} $. 
\end{lemma}

Lemmas \ref{lem:2} and \ref{lem:3} now combine to establish the limiting distribution of the test statistic to be the remarkably simple chi-squared distribution, see Proposition \ref{prop:1} in Section \ref{sec:main}.


\section{Simulations}\label{sec:simulations}

The following results are all obtained in R \citep{R} using the packages JADE \citep{JADE} and tsBSS \citep{tsBSS}.

\subsection{Evaluation of the hypothesis testing}\label{sec:simulations_1}

In the first set of simulations we consider the performance of the hypothesis tests. As our competitor we use the recommended and most general non-parametric bootstrapping strategy from  \citet{MatilainenNordhausenVirta2017}, which takes bootstrap samples from the hypothetical multivariate noise part. The number of bootstrap samples used was 200. We computed also the three other bootstrapping strategies as listed in \citet{MatilainenNordhausenVirta2017}, but as the results were basically the same, we report for simplicity only the strategy mentioned above.

We considered three different settings for the latent sources:

\begin{description}
	\item[Setting H1:] MA(3), AR(2) and ARMA(1,1) having Gaussian innovations together with two Gaussian white noise components.
	\item[Setting H2:] MA(10), MA(15) and M(20) processes having Gaussian innovations together with two Gaussian white noise components.
	\item[Setting H3:] Three MA(3) processes having Gaussian innovations and identical autocovariance functions together with two Gaussian white noise processes.
\end{description}

Hence, in all three settings the signal dimension is $d=3$ and the total dimension is $p=5$. Due to affine equivariance of the used methods, we take without loss of generality $\bs \Omega = \bo I_5$. In general, setting H1 can be considered a short range dependence model and H2 a long range dependence model. Setting H3 is special in that the methods should not be able to separate its signals, but they should still be able to separate the noise space from the signal space.

Based on 2000 repetitions, we give the rejection frequencies of the null hypotheses $ H_{02} $, $ H_{03} $ and $ H_{04} $ at level $\alpha=0.05$ in Tables~\ref{tab:H1d4}-\ref{tab:H3d2}. We considered three different BSS-estimators, AMUSE with $\tau=1$, SOBI with $\mathcal{T}=\{1,\ldots,6\}$ (denoted SOBI6) and 
SOBI with $\mathcal{T}=\{1,\ldots,12\}$ (denoted SOBI12). The optimal rejection rates at level $ \alpha = 0.05 $ are $ 1.00 $ for $ H_{02} $, $ 0.05 $ for $ H_{03} $ and $ < 0.05 $ for $ H_{04} $.

\begin{longtable}[]{@{}rrrrrrr@{}}
		\caption{Rejection frequencies of $H_{02}$ in Setting H1 at level $\alpha=0.05$.}
	\label{tab:H1d4}\tabularnewline
	\toprule
	& \multicolumn{2}{c}{AMUSE}  & \multicolumn{2}{c}{SOBI6} & \multicolumn{2}{c}{SOBI12}\tabularnewline \cmidrule(l){2-3}  \cmidrule(l){4-5} \cmidrule(l){6-7}
		n & Asymp & Boot & Asymp & Boot & Asymp & Boot\tabularnewline
	\midrule
	\endfirsthead
	200 & 1.000 & 1.000 & 1.000 & 0.999 & 0.998 & 0.998\tabularnewline
	500 & 1.000 & 1.000 & 1.000 & 1.000 & 1.000 & 1.000\tabularnewline
	1000 & 1.000 & 1.000 & 1.000 & 1.000 & 1.000 & 1.000\tabularnewline
	2000 & 1.000 & 1.000 & 1.000 & 1.000 & 1.000 & 1.000\tabularnewline
	5000 & 1.000 & 1.000 & 1.000 & 1.000 & 1.000 & 1.000\tabularnewline
	\bottomrule
\end{longtable}

\begin{longtable}[]{@{}rcccccc@{}}
	\caption{Rejection frequencies when testing \(H_{03}\) in Setting H1 at level $\alpha=0.05$.}
	\label{tab:H1d3}\tabularnewline
	\toprule
	& \multicolumn{2}{c}{AMUSE}  & \multicolumn{2}{c}{SOBI6} & \multicolumn{2}{c}{SOBI12}\tabularnewline \cmidrule(l){2-3}  \cmidrule(l){4-5} \cmidrule(l){6-7}
		n & Asymp & Boot & Asymp & Boot & Asymp & Boot\tabularnewline
	\midrule
	\endfirsthead
	200 & 0.059 & 0.050 & 0.078 & 0.050 & 0.102 & 0.050\tabularnewline
	500 & 0.053 & 0.048 & 0.064 & 0.049 & 0.071 & 0.052\tabularnewline
	1000 & 0.048 & 0.047 & 0.059 & 0.053 & 0.054 & 0.050\tabularnewline
	2000 & 0.050 & 0.054 & 0.048 & 0.049 & 0.054 & 0.046\tabularnewline
	5000 & 0.048 & 0.052 & 0.052 & 0.047 & 0.056 & 0.053\tabularnewline
	\bottomrule
\end{longtable}

\begin{longtable}[]{@{}rcccccc@{}}
	\caption{Rejection frequencies when testing $H_{04}$ in Setting H1 at level $\alpha=0.05$.}
\label{tab:H1d2}\tabularnewline
	\toprule
	& \multicolumn{2}{c}{AMUSE}  & \multicolumn{2}{c}{SOBI6} & \multicolumn{2}{c}{SOBI12}\tabularnewline \cmidrule(l){2-3}  \cmidrule(l){4-5} \cmidrule(l){6-7}
		n & Asymp & Boot & Asymp & Boot & Asymp & Boot\tabularnewline
	\midrule
	\endfirsthead
	200 & 0.006 & 0.008 & 0.015 & 0.006 & 0.024 & 0.004\tabularnewline
	500 & 0.006 & 0.007 & 0.009 & 0.004 & 0.016 & 0.006\tabularnewline
	1000 & 0.007 & 0.010 & 0.012 & 0.005 & 0.012 & 0.006\tabularnewline
	2000 & 0.003 & 0.006 & 0.008 & 0.003 & 0.009 & 0.002\tabularnewline
	5000 & 0.006 & 0.006 & 0.006 & 0.002 & 0.008 & 0.004\tabularnewline
	\bottomrule
\end{longtable}

\begin{longtable}[]{@{}rcccccc@{}}
		\caption{Rejection frequencies when testing \(H_{02}\) in Setting H2 at level $\alpha=0.05$.}
	\label{tab:H2d4}\tabularnewline
	\toprule
	& \multicolumn{2}{c}{AMUSE}  & \multicolumn{2}{c}{SOBI6} & \multicolumn{2}{c}{SOBI12}\tabularnewline \cmidrule(l){2-3}  \cmidrule(l){4-5} \cmidrule(l){6-7}
		n & Asymp & Boot & Asymp & Boot & Asymp & Boot\tabularnewline
	\midrule
	\endfirsthead
	200 & 0.038 & 0.043 & 0.608 & 0.484 & 0.911 & 0.848\tabularnewline
	500 & 0.090 & 0.094 & 0.988 & 0.987 & 1.000 & 1.000\tabularnewline
	1000 & 0.190 & 0.189 & 1.000 & 1.000 & 1.000 & 1.000\tabularnewline
	2000 & 0.252 & 0.256 & 1.000 & 1.000 & 1.000 & 1.000\tabularnewline
	5000 & 0.558 & 0.550 & 1.000 & 1.000 & 1.000 & 1.000\tabularnewline
	\bottomrule
\end{longtable}

\begin{longtable}[]{@{}rcccccc@{}}
	\caption{Rejection frequencies when testing \(H_{03}\) in Setting H2 at level $\alpha=0.05$.}
	\label{tab:H2d3}\tabularnewline
	\toprule
	& \multicolumn{2}{c}{AMUSE}  & \multicolumn{2}{c}{SOBI6} & \multicolumn{2}{c}{SOBI12}\tabularnewline \cmidrule(l){2-3}  \cmidrule(l){4-5} \cmidrule(l){6-7}
	n & Asymp & Boot & Asymp & Boot & Asymp & Boot\tabularnewline
	\midrule
	\endfirsthead
	200 & 0.002 & 0.006 & 0.125 & 0.050 & 0.148 & 0.063\tabularnewline
	500 & 0.008 & 0.014 & 0.075 & 0.041 & 0.074 & 0.050\tabularnewline
	1000 & 0.010 & 0.014 & 0.067 & 0.046 & 0.068 & 0.047\tabularnewline
	2000 & 0.020 & 0.024 & 0.056 & 0.052 & 0.066 & 0.061\tabularnewline
	5000 & 0.031 & 0.039 & 0.051 & 0.048 & 0.054 & 0.047\tabularnewline
	\bottomrule
\end{longtable}

\begin{longtable}[]{@{}rcccccc@{}}
	\caption{Rejection frequencies when testing \(H_{04}\) in Setting H2 at level $\alpha=0.05$.}
	\label{tab:H2d2}\tabularnewline
	\toprule
	& \multicolumn{2}{c}{AMUSE}  & \multicolumn{2}{c}{SOBI6} & \multicolumn{2}{c}{SOBI12}\tabularnewline \cmidrule(l){2-3}  \cmidrule(l){4-5} \cmidrule(l){6-7}
	n & Asymp & Boot & Asymp & Boot & Asymp & Boot\tabularnewline
	\midrule
	\endfirsthead
	200 & 0.000 & 0.001 & 0.034 & 0.004 & 0.039 & 0.006\tabularnewline
	500 & 0.002 & 0.004 & 0.010 & 0.004 & 0.016 & 0.007\tabularnewline
	1000 & 0.000 & 0.004 & 0.012 & 0.004 & 0.007 & 0.001\tabularnewline
	2000 & 0.002 & 0.004 & 0.010 & 0.004 & 0.010 & 0.003\tabularnewline
	5000 & 0.004 & 0.008 & 0.010 & 0.005 & 0.007 & 0.003\tabularnewline
	\bottomrule
\end{longtable}

\begin{longtable}[]{@{}rcccccc@{}}
	\caption{Rejection frequencies when testing \(H_{02}\) in Setting H3 at level $\alpha=0.05$.}
	\label{tab:H3d4}\tabularnewline
	\toprule
	& \multicolumn{2}{c}{AMUSE}  & \multicolumn{2}{c}{SOBI6} & \multicolumn{2}{c}{SOBI12}\tabularnewline \cmidrule(l){2-3}  \cmidrule(l){4-5} \cmidrule(l){6-7}
	n & Asymp & Boot & Asymp & Boot & Asymp & Boot\tabularnewline
	\midrule
	\endfirsthead
	200 & 0.036 & 0.042 & 0.600 & 0.479 & 0.906 & 0.84\tabularnewline
	500 & 0.084 & 0.092 & 0.986 & 0.987 & 1.000 & 1.00\tabularnewline
	1000 & 0.168 & 0.175 & 1.000 & 1.000 & 1.000 & 1.00\tabularnewline
	2000 & 0.279 & 0.272 & 1.000 & 1.000 & 1.000 & 1.00\tabularnewline
	5000 & 0.576 & 0.568 & 1.000 & 1.000 & 1.000 & 1.00\tabularnewline
	\bottomrule
\end{longtable}

\begin{longtable}[]{@{}rcccccc@{}}
	\caption{Rejection frequencies when testing \(H_{03}\) in Setting H3 at level $\alpha=0.05$.}
	\label{tab:H3d3}\tabularnewline
	\toprule
	& \multicolumn{2}{c}{AMUSE}  & \multicolumn{2}{c}{SOBI6} & \multicolumn{2}{c}{SOBI12}\tabularnewline \cmidrule(l){2-3}  \cmidrule(l){4-5} \cmidrule(l){6-7}
	n & Asymp & Boot & Asymp & Boot & Asymp & Boot\tabularnewline
	\midrule
	\endfirsthead
	200 & 0.004 & 0.005 & 0.122 & 0.049 & 0.146 & 0.047\tabularnewline
	500 & 0.006 & 0.008 & 0.075 & 0.043 & 0.074 & 0.057\tabularnewline
	1000 & 0.010 & 0.018 & 0.062 & 0.050 & 0.062 & 0.055\tabularnewline
	2000 & 0.016 & 0.023 & 0.058 & 0.044 & 0.046 & 0.054\tabularnewline
	5000 & 0.034 & 0.042 & 0.051 & 0.050 & 0.048 & 0.045\tabularnewline
	\bottomrule
\end{longtable}

\begin{longtable}[]{@{}rcccccc@{}}
	\caption{Rejection frequencies when testing \(H_{04}\) in Setting H3 at level $\alpha=0.05$.}
	\label{tab:H3d2}\tabularnewline
	\toprule
	& \multicolumn{2}{c}{AMUSE}  & \multicolumn{2}{c}{SOBI6} & \multicolumn{2}{c}{SOBI12}\tabularnewline \cmidrule(l){2-3}  \cmidrule(l){4-5} \cmidrule(l){6-7}
	n & Asymp & Boot & Asymp & Boot & Asymp & Boot\tabularnewline
	\midrule
	\endfirsthead
	200 & 0.000 & 0.002 & 0.026 & 0.005 & 0.034 & 0.006\tabularnewline
	500 & 0.000 & 0.002 & 0.012 & 0.003 & 0.012 & 0.003\tabularnewline
	1000 & 0.002 & 0.002 & 0.010 & 0.005 & 0.010 & 0.005\tabularnewline
	2000 & 0.004 & 0.006 & 0.008 & 0.007 & 0.006 & 0.004\tabularnewline
	5000 & 0.005 & 0.010 & 0.008 & 0.004 & 0.006 & 0.003\tabularnewline
	\bottomrule
\end{longtable}

The results of the simulations can be summarized as follows. (i) There is no big difference between the limiting theory and the bootstrap tests, which is a clear advantage for the asymptotic test as neither a bootstrapping strategy has to be selected nor is the asymptotic test computationally demanding. (ii) The number of matrices to be diagonalized seems to matter. If the dependence structure is of a short range AMUSE works well, but it seems to struggle in the case of long range dependence. In the considered settings SOBI with 6 matrices seems to be a good compromise. (iii) Even when the signals cannot be individually separated, the noise and signal subspaces can be separated very accurately.

In general, having a very good power under the alternative hypotheses of too large noise subspaces is desirable when using successive testing strategies to estimate the dimension. This was not yet evaluated in \citet{MatilainenNordhausenVirta2017} and will be done in the next section. 

\subsection{Evaluation of determining the dimension of the signal}

In this section we evaluate in a simulation study the performance of our test when the goal is to estimate the signal dimension $d$. Several different testing strategies are possible, as described in the end of Section \ref{sec:main}. We will use in the following the divide-and-conquer strategy as it seems the most practical. For simplicity, all tests will be performed at the level $\alpha=0.05$.

As competitors we use again the bootstrap tests, this time including all three nonparametric bootstraps and the parametric bootstrap. For details we refer to \citet{MatilainenNordhausenVirta2017}. As an additional contender we use the ladle estimator as described in \citet{NordhausenVirta2018}. Also for the ladle different bootstrap strategies are possible and we consider the fixed block bootstrap with the block lengths 20 and 40 and the stationary block bootstrap with the expected block lengths 20 and 40, see \citet{NordhausenVirta2018} for details. For all bootstrap-based methods the number of bootstrapping samples is again 200 and, as in the previous section, we consider the three estimators, AMUSE, SOBI6 and SOBI12.

The settings considered in this simulation are: 
\begin{description}
	\item[Setting D1:] AR(2), AR(3), ARMA(1,1), ARMA(3,2) and MA(3) processes having Gaussian innovations  together with 5 Gaussian white noise components.
	\item[Setting D2:] Same processes as in D1 but the MA(3) is changed to an MA(1) process with the parameter equal to 0.1.
	\item[Setting D3:] Five MA(2) processes with parameters (0.1, 0.1) having Gaussian innovations together with 5 Gaussian white noise processes.
\end{description}

Hence, in all settings $p=10$ and $d=5$. Setting D1 is the basic setting whereas in Setting D2 there is one very weak signal. In Setting D3 all five signals come from identical processes and exhibit weak dependence. As in the previous simulation, the mixing matrix used is $\bs \Omega = \bo I_{10}$ and Figures~\ref{fig:Est_k_1}-\ref{fig:Est_k_3} show, based on 2000 repetitions, the frequencies of the estimated signal dimensions.

\begin{figure}[t!]
	\centering
	\includegraphics[width=0.99\textwidth]{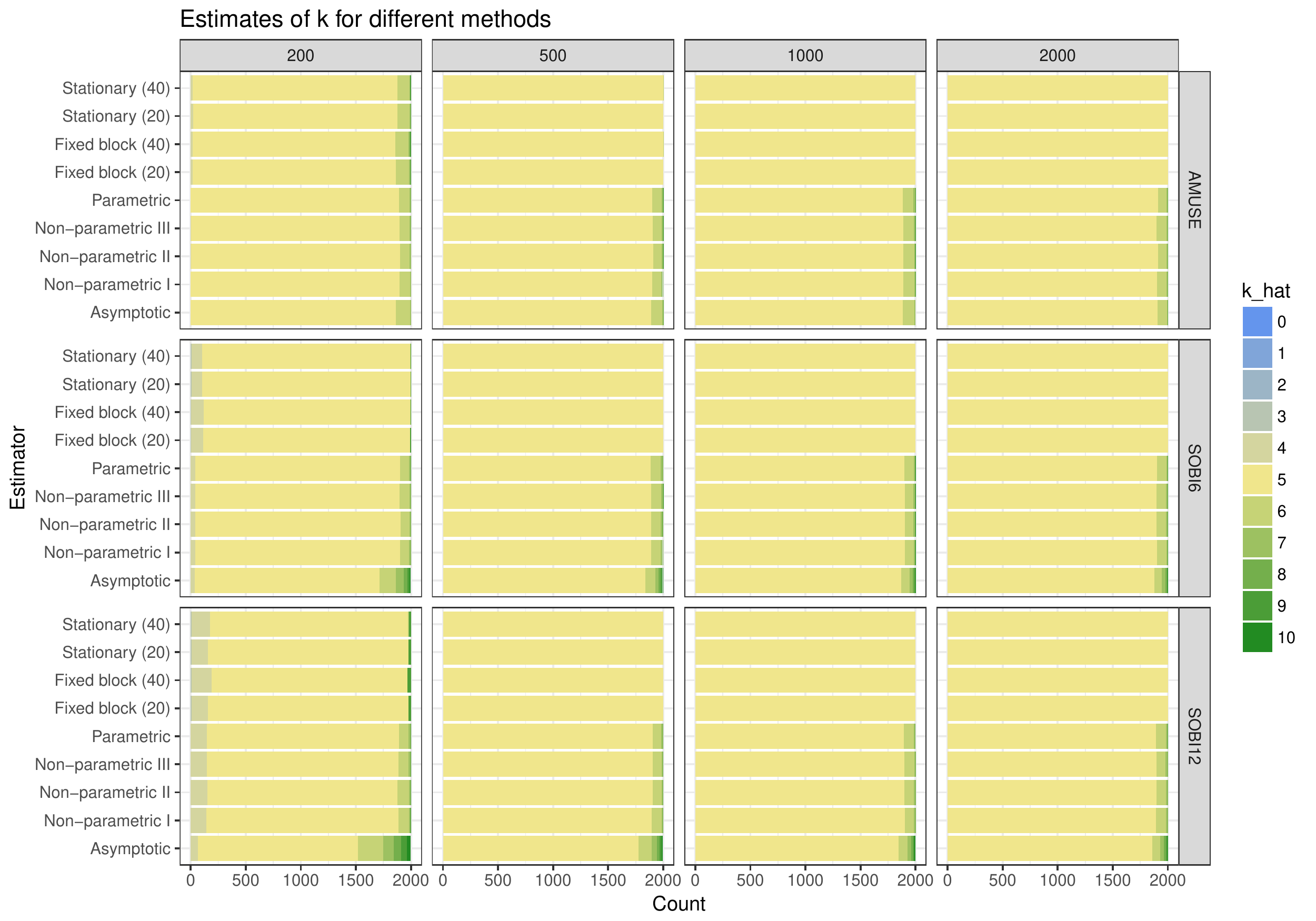}
	\caption{Estimating $k$ by divide-and-conquer in Setting D1.}
	\label{fig:Est_k_1}
\end{figure}

\begin{figure}[t!]
	\centering
	\includegraphics[width=0.99\textwidth]{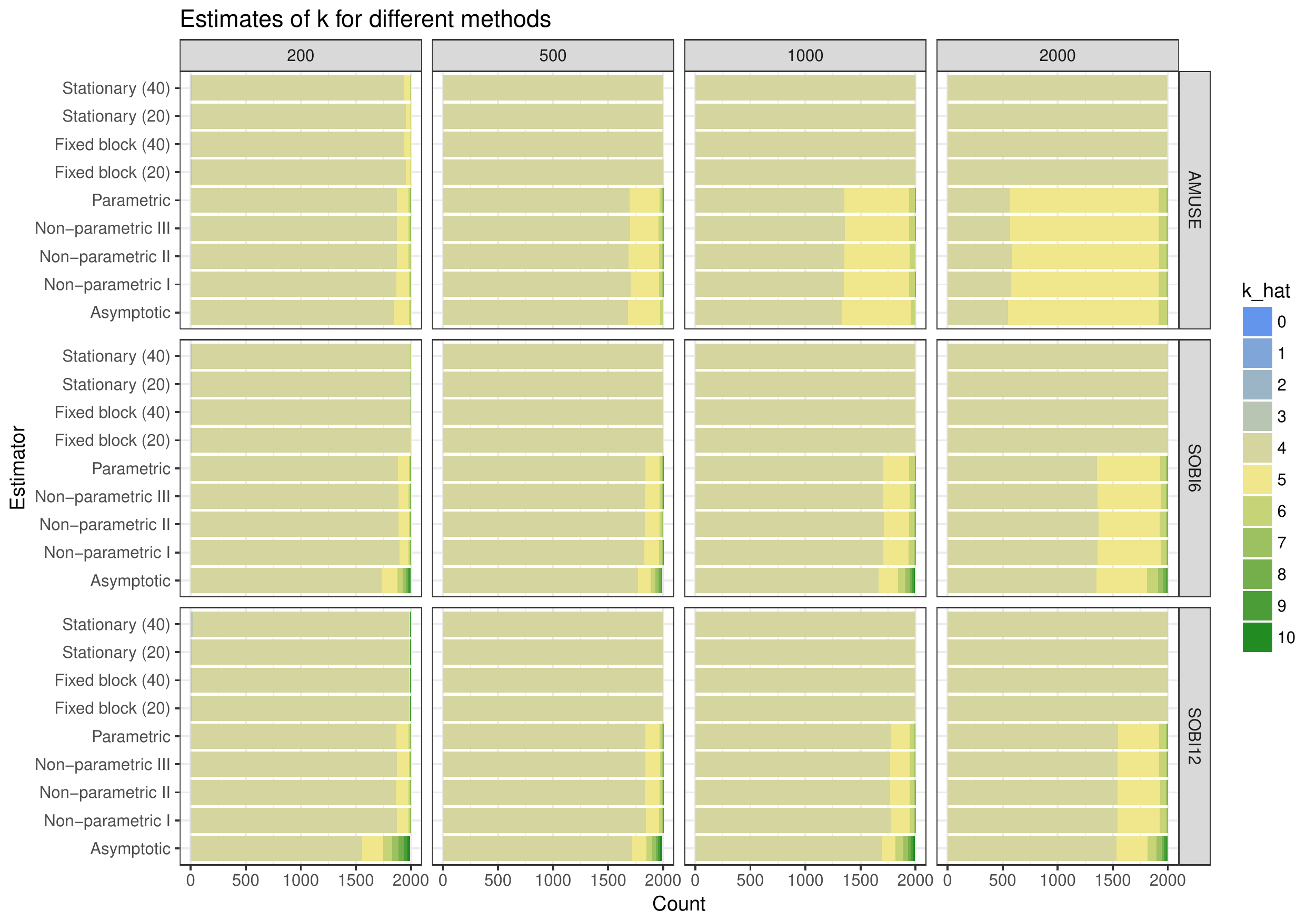}
	\caption{Estimating $k$ by divide-and-conquer in Setting D2.}
	\label{fig:Est_k_2}
\end{figure}

\begin{figure}[t!]
	\centering
	\includegraphics[width=0.99\textwidth]{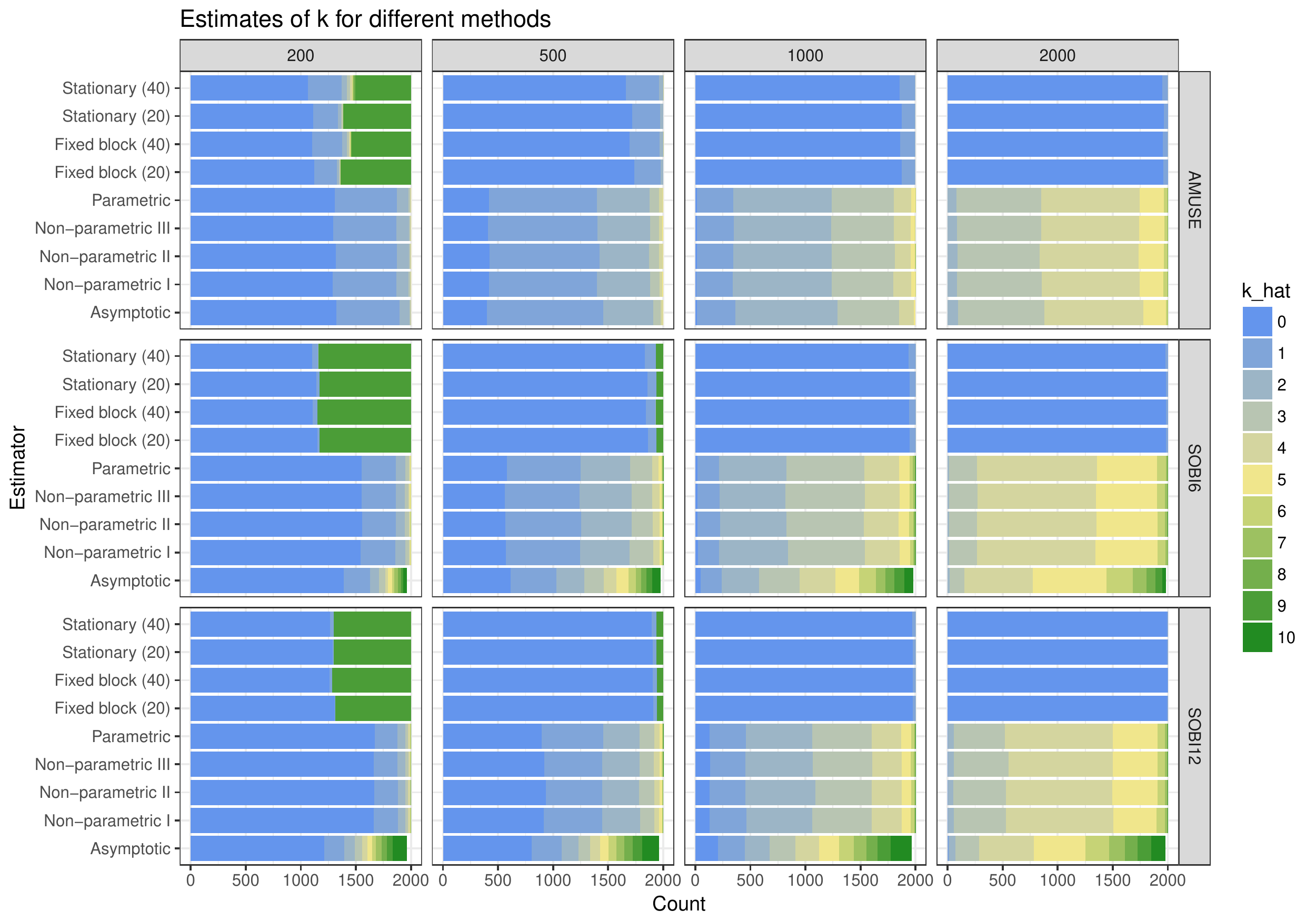}
	\caption{Estimating $k$ by divide-and-conquer in Setting D3.}
	\label{fig:Est_k_3}
\end{figure}

In Setting D1, the asymptotic test seems not to work as well as the other methods for small samples but in general the difference to the bootstrap-based testing procedures is negligible. In general, the ladle is the most preferable option. In setting D2, on the other hand, ladle consistently underestimates the signal dimension and is not able to find the weak signal. When using the hypothesis testing-based methods also the weak signal is identified with increasing sample size. However, the more scatter matrices we estimate, the more difficult the estimation gets and thus AMUSE works the best.

In Setting D3 the ladle fails completely and keeps getting worse with increasing sample size. The difference between bootstrapping and asymptotic testing is at its largest in this setting, the asymptotic test seems to be the most preferable option. As two lags are needed to capture all the temporal information, AMUSE is at an disadvantage in this setting, and this is clearly visible in the plots. Also, SOBI6 seems to exploit the lag information better than SOBI12, possibly because it avoids the inclusion of several unnecessary autocovariance matrices in the estimation.

\section{Data example}\label{sec:data_example}

For our real data example we use the recordings of three sounds signals available in the R-package JADE and analysed, for example, in \citet{MiettinenNordhausenTaskinen2017}. To the three signal components we added 17 white noise components which all had $t_5$-distributions to study whether the methods also work in case of non-Gaussian white noise. After standardizing the 20 components to have unit variances, we used a random square matrix where each element came from the uniform distribution on $[0,1]$. The original signals had a length of 50000 and for convenience we selected the first 10000 instances. The 20 mixed components are visualized in Figure~\ref{fig:DataSoundExample} and reveal no clear structure.

\begin{figure}[H]
	\centering
	\includegraphics[width=0.99\textwidth]{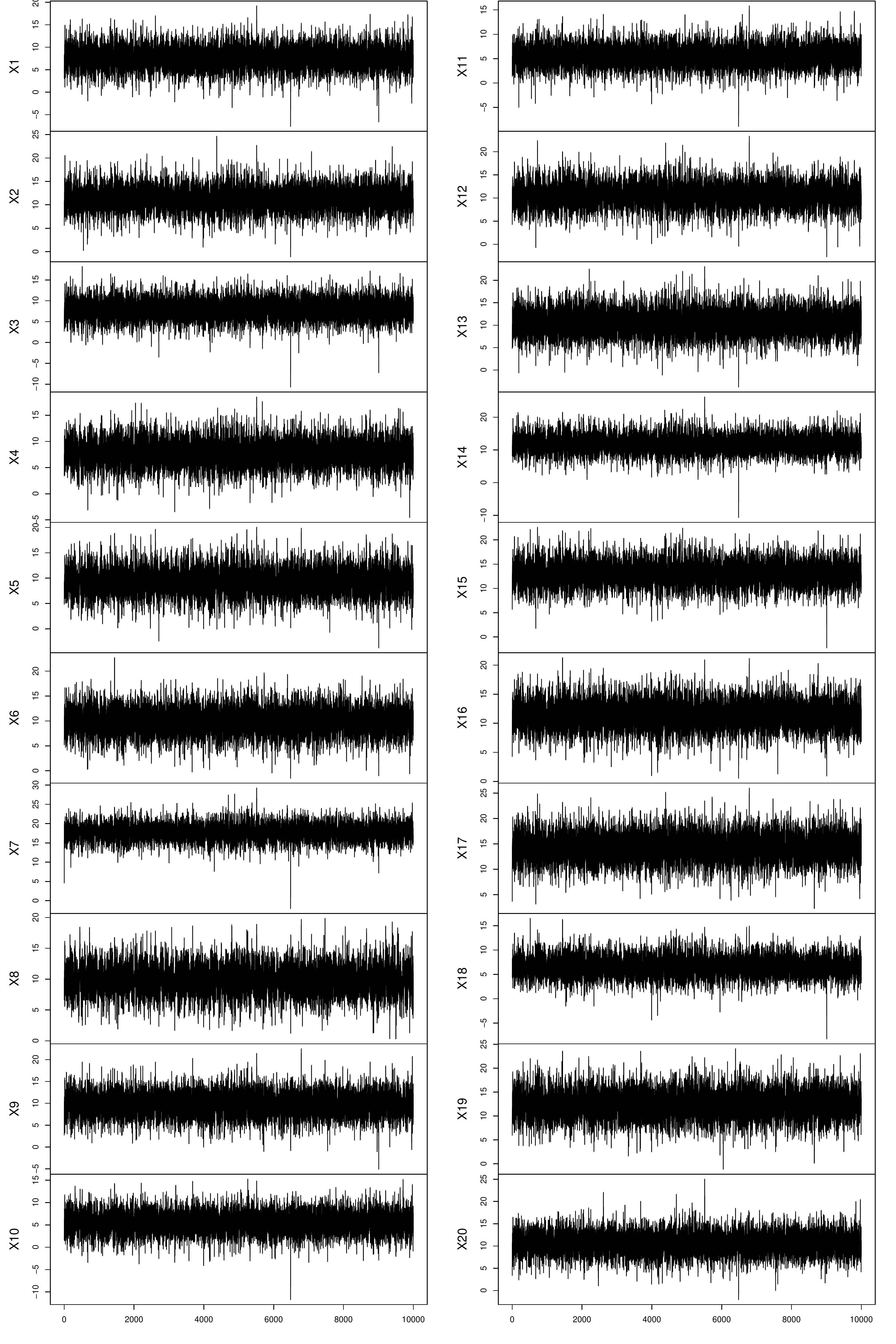}
	\caption{The 20-variate sound data time series.}
	\label{fig:DataSoundExample}
\end{figure}

We used the divide-and-conquer approach to estimate the signal dimension with our asymptotic test and the bootstrapping strategy of \citet{MatilainenNordhausenVirta2017} used in Section~\ref{sec:simulations_1}. Additionally, we considered also the ladle estimator using stationary bootstrapping with the expected block length 40. Of each estimator, three versions, AMUSE, SOBI6 and SOBI12, were computed. All nine estimators estimated correctly 3 as the signal dimension and the estimated signals based on SOBI6 are shown in Figure~\ref{fig:EstimatedSignalsSoundExample}.

\begin{figure}[t!]
	\centering
	\includegraphics[width=0.8\textwidth]{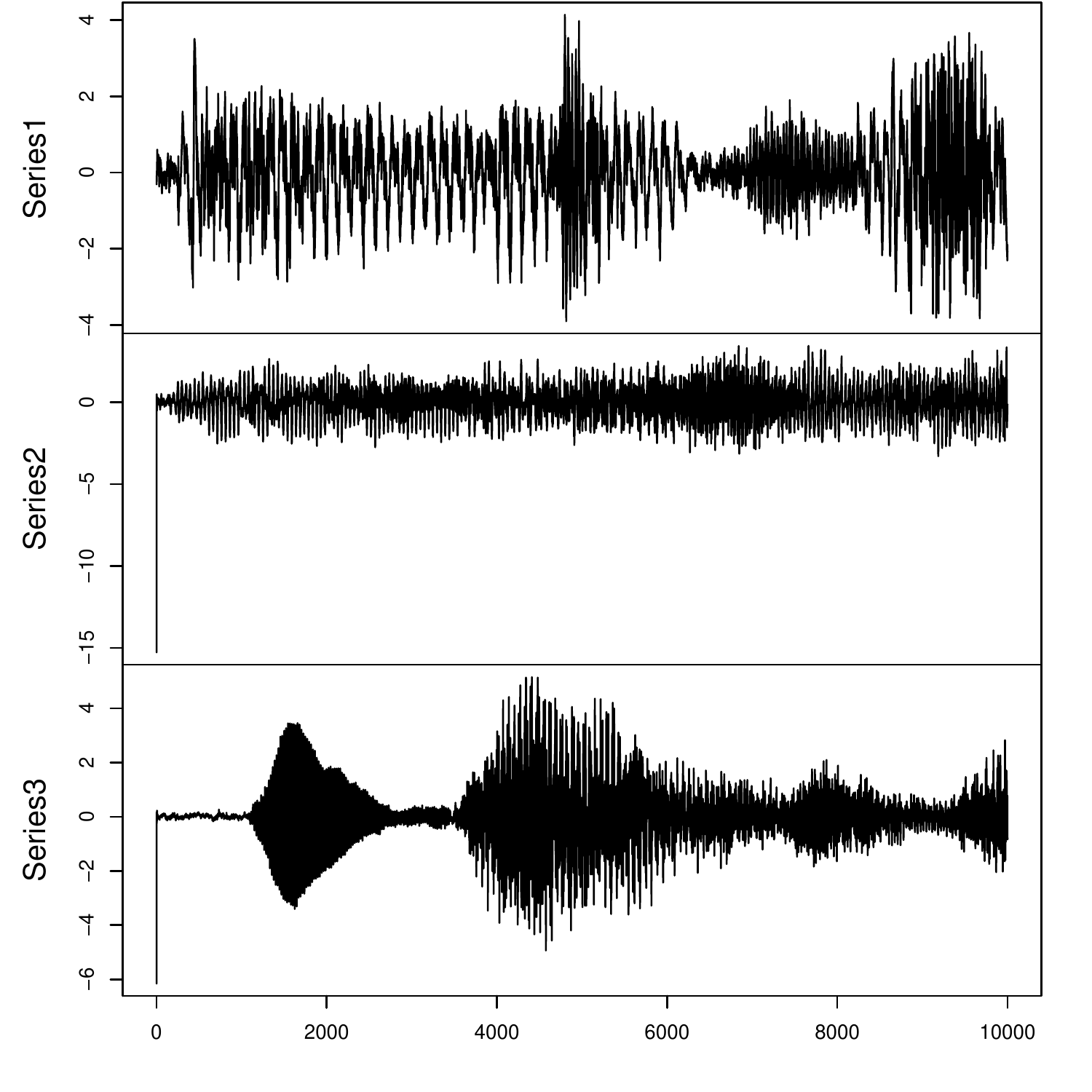}
	\caption{The three estimated sound signals based on SOBI6.}
	\label{fig:EstimatedSignalsSoundExample}
\end{figure}

The computation times of the nine methods were, however, quite different and are given in Table~\ref{tab:SoundExTimings}. 

\begin{table}[ht]
	\centering
	\footnotesize
	\caption{Computation times (in seconds) of the nine estimators for the sound example data.} \label{tab:SoundExTimings}
	\begin{tabular}{ccccccccc}
	\toprule
	  \multicolumn{3}{c}{Asymptotic tests}  & \multicolumn{3}{c}{Bootstrap tests} & \multicolumn{3}{c}{Ladle estimator}\tabularnewline \cmidrule(l){1-3}  \cmidrule(l){4-6} \cmidrule(l){7-9}
		 AMUSE &SOBI6 & SOBI12 & AMUSE &SOBI6 & SOBI12 & AMUSE &SOBI6 & SOBI12 \\ 
		\midrule
		 0.07 & 0.19 & 0.49 & 15.08 & 47.24 & 88.08 & 2.75 & 9.85 & 18.17 \\ 
		\bottomrule
	\end{tabular}
\end{table}

As all the approaches estimated the dimension correctly, the ones based on the asymptotic test are clearly favourable due to their computational speed. Although, we note that the ladle and the bootstrap tests can also be run parallelized while in the current comparison we used only a single core.

\section{Discussion}\label{sec:discussion}

We proposed an asymptotic test for estimating the signal dimension in an SOS-model where the sources include both signal series and white noise. The test does not require the estimation of any parameters and makes quite weak assumptions. This combined with the fact that both of its existing competitors are based on the use of computer-intensive resampling techniques makes the asymptotic test a very attractive option in practice. This conclusion was supported by our simulations studies and real data example exploring dimension estimation for sound recording data.

A central drawback of the proposed method is its inability to recognize non-autocorrelated signals, such as those coming from stochastic volatility models, from white noise. One way to get around this limitation is to replace $ \textbf{z}_{2t} $ in \eqref{eq:model} by a vector of stochastic volatility series and to revert the roles of signal and noise. That is, we use hypothesis testing to estimate the dimension of the ``noise'' subspace (containing the stochastic volatility components) and separate them from the uninteresting ``signal'' series exhibiting autocorrelation. For this to work, a limiting result equivalent to Lemma \ref{lem:3} is needed for the above combination model. Similar idea was suggested in the context of the ladle estimator already in \cite{NordhausenVirta2018}.

%
\vskip 14pt
\noindent {\large\bf Acknowledgements}

Klaus Nordhausen acknowledges support from the CRoNoS COST Action IC1408. 
\par


\appendix

\section{Proofs}\label{sec:appendix}


\begin{proof}[Proof of Lemma \ref{lem:0}]
	By Assumption \ref{assu:consistent} we have,
	\[ 
	\sqrt{T} ( \hat{\textbf{S}}_\tau - \textbf{D}_\tau) = \mathcal{O}_p(1), \quad \mbox{for all } \tau \in \mathcal{T} \cup \{ 0 \},
	\]
	where $ \textbf{D}_0 = \textbf{I}_p $. This instantly implies the equivalent result for the symmetrized autocovariance matrices $ \hat{\textbf{R}}_\tau $,
	\[ 
	\sqrt{T} ( \hat{\textbf{R}}_\tau - \textbf{D}_\tau) = \mathcal{O}_p(1), \quad \mbox{for all } \tau \in \mathcal{T}.
	\]
	Let $ \mathrm{vect} $ be the row-vectorization operator that takes the row vectors of a matrix and stacks them into a long column. That is, $ \mathrm{vect}(\textbf{X}) \in \mathbb{R}^{mn}$ for any $ \textbf{A} \in \mathbb{R}^{m \times n}$ and $\mathrm{vect}(\textbf{A} \textbf{X} \textbf{B}) = (\textbf{A} \otimes \textbf{B}) \mathrm{vect}(\textbf{X})$ for any $ \textbf{A} \in \mathbb{R}^{s \times m}, \textbf{A} \in \mathbb{R}^{m \times n}, \textbf{B} \in \mathbb{R}^{n \times t}$. By linearizing and row-vectorizing the definition $ \textbf{0} = \sqrt{T} (\hat{\textbf{S}}{}_0^{-1/2} \hat{\textbf{S}}{}_0 \hat{\textbf{S}}{}_0^{-1/2} - \textbf{I}_p) $ and using Slutsky's theorem, we obtain,
	\[ 
	\hat{\textbf{B}} \sqrt{T} \mathrm{vect} ( \hat{\textbf{S}}{}_0^{-1/2} - \textbf{I}_p ) = - \sqrt{T} \mathrm{vect} ( \hat{\textbf{S}}{}_0 - \textbf{I}_p ) + o_p(1),
	\]
	where $ \hat{\textbf{B}} =  \textbf{I}_p  \otimes \hat{\textbf{S}}{}_0 \hat{\textbf{S}}{}_0^{-1/2} + \textbf{I}_p \otimes \textbf{I}_p $. As $ \hat{\textbf{B}} \rightarrow_p \textbf{I}_{p^2} $, its inverse is asymptotically well-defined, allowing us to multiply the relation from the left by $ \hat{\textbf{B}}{}^{-1} $, after which Slutsky's theorem and Assumption \ref{assu:consistent} yield that $ \sqrt{T} ( \hat{\textbf{S}}{}_0^{-1/2} - \textbf{I}_p ) = \mathcal{O}_p(1) $.
	
	Linearize next as,
	\begin{align}\label{eq:H_linearization}
	\begin{split}
	\sqrt{T} ( \hat{\textbf{H}}_\tau - \textbf{D}_\tau) &= \sqrt{T} ( \hat{\textbf{S}}{}_0^{-1/2} - \textbf{I}_p ) \hat{\textbf{R}}_\tau \hat{\textbf{S}}{}_0^{-1/2} + \sqrt{T} ( \hat{\textbf{R}}_\tau - \textbf{D}_\tau ) \hat{\textbf{S}}{}_0^{-1/2}\\
	&+ \textbf{D}_\tau \sqrt{T} ( \hat{\textbf{S}}{}_0^{-1/2} - \textbf{I}_p ),
	\end{split}
	\end{align}
	where the right-hand side is by the previous convergence results expressible as $ \sqrt{T} ( \hat{\textbf{R}}_\tau - \textbf{D}_\tau ) + \mathcal{O}_p(1) $. The first claim now follows after the division by $ \sqrt{T} $ and the addition of $ \textbf{D}_\tau $ on both sides.
	
	For the second claim we observe only the lower right $ p_0 \times p_0 $ corner block $ \hat{\textbf{H}}_{\tau00} $ and write,
	\begin{align}\label{eq:H00_linearization} 
	\hat{\textbf{H}}_{\tau00} = \hat{\textbf{T}}_1^\top \hat{\textbf{R}}_{\tau,-0} \hat{\textbf{T}}_1 + \hat{\textbf{T}}_2^\top \hat{\textbf{R}}_{\tau00} \hat{\textbf{T}}_2,  
	\end{align}
	where $ (\hat{\textbf{T}}{}_1^\top, \hat{\textbf{T}}{}_2^\top)^\top $, $ \hat{\textbf{T}}{}_1 \in \mathbb{R}^{(p - p_0) \times p_0} $, $ \hat{\textbf{T}}{}_2 \in \mathbb{R}^{p_0 \times p_0} $ denotes the final $ p \times p_0 $ column block of $  \hat{\textbf{S}}{}_0^{-1/2} $ and $ \hat{\textbf{R}}_{\tau,-0} \in \mathbb{R}^{(p - p_0) \times (p - p_0)} $ denotes the result of removing the final $ p_0 $ rows and columns from $ \hat{\textbf{R}}_\tau $. These matrices satisfy $ \hat{\textbf{T}}_1 = \mathcal{O}_p(1/\sqrt{T}) $, $ \hat{\textbf{T}}_2 - \textbf{I}_{p_0} =  \mathcal{O}_p(1/\sqrt{T}) $ and  $ \hat{\textbf{R}}_{\tau00} = \mathcal{O}_p(1/\sqrt{T}) $ and we can write \eqref{eq:H00_linearization} as
	\begin{align*} 
	\hat{\textbf{H}}_{\tau00} &= \hat{\textbf{T}}_1^\top \hat{\textbf{R}}_{\tau,-0} \hat{\textbf{T}}_1 + ( \hat{\textbf{T}}_2 - \textbf{I}_{p_0} )^\top \hat{\textbf{R}}_{\tau00} ( \hat{\textbf{T}}_2 - \textbf{I}_{p_0} ) + \hat{\textbf{R}}_{\tau00} ( \hat{\textbf{T}}_2 - \textbf{I}_{p_0} ) \\
	&+ ( \hat{\textbf{T}}_2 - \textbf{I}_{p_0} )^\top \hat{\textbf{R}}_{\tau00} + \hat{\textbf{R}}_{\tau00} \\
	&= \hat{\textbf{R}}_{\tau00} + \mathcal{O}_p(1/T),
	\end{align*}
	concluding the proof.
\end{proof}

\begin{proof}[Proof of Lemma \ref{lem:1}]
	The SOBI-solution is found as $ \hat{\textbf{U}}{}^\top \hat{\textbf{S}}{}_{0}^{-1/2}$ where the orthogonal matrix $ \hat{\textbf{U}} $ is the maximizer of
	\begin{align}\label{eq:sobi_problem} 
	\hat{g}(\textbf{U}) = \sum_{\tau \in \mathcal{T}} \left\| \mathrm{diag} \left( \textbf{U}^\top \hat{\textbf{H}}_\tau \textbf{U}  \right) \right\|^2.
	\end{align}
	Let $ \hat{\textbf{U}} $ be a sequence of maximizers of \eqref{eq:sobi_problem} and partition $ \hat{\textbf{U}} $ in the blocks $ \hat{\textbf{U}}_{ij} $ in a similar way as in the problem statement (ignoring the sequence of permutations $ \hat{\textbf{P}} $ for now). The proof of the lemma is divided into two parts. First, we establish the consistency of the off-diagonal blocks, $ \hat{\textbf{U}}_{ij} \rightarrow_p \textbf{0} $, and, second, we show the rate of convergence, $\sqrt{T}\hat{\textbf{U}}_{ij} = \mathcal{O}_p( 1 )$. That the diagonal blocks $ \hat{\textbf{U}}_{ii} $ are stochastically bounded follows simply from the compactness of the space of orthogonal matrices.
	
	\medskip
	
	\noindent{\small\textbf{1. Consistency}}
	
	\medskip
	
	Our aim to is to use a technique similar to the $ M $-estimator consistency argument \citep[Theorem 5.7]{van1998asymptotic}, for which we need the Fisher consistency of the off-diagonal blocks, along with the uniform convergence of the sample objective function to the population objective function with respect to $ \textbf{U} $. By Fisher consistency we mean that all maximizers $ \textbf{U} $ of the population objective function, 
	\[ 
	g(\textbf{U}) = \sum_{\tau \in \mathcal{T}} \left\| \mathrm{diag} \left( \textbf{U}^\top  \textbf{H}_{\tau}  \textbf{U}  \right) \right\|^2,
	\]
	where $ \textbf{H}_\tau = \textbf{S}{}_{0}^{-1/2} \textbf{R}_{\tau} \textbf{S}{}_{0}^{-1/2} $, can have their columns ordered to satisfy $ \textbf{U}_{ij} = \textbf{0} $ for all $ i \neq j $.
	
	The population autocovariance matrices satisfy,
	\[ 
	\textbf{S}_{0} = \textbf{I}_p \quad \mbox{and} \quad \textbf{H}_{\tau} = \textbf{R}_{\tau} = \textbf{S}_{\tau} = \textbf{D}_\tau,
	\]
	where $ \textbf{D}_\tau = \mathrm{diag}(\lambda_{\tau 1}\textbf{I}_{p_1} , \ldots , \lambda_{\tau v} \textbf{I}_{p_v}, \textbf{0}) \in \mathbb{R}^{p \times p}$ are diagonal matrices, $ \tau \in \mathcal{T} $.
	The population objective function has the upper bound,
	\begin{align}\label{eq:population_objective}
	g(\textbf{U}) = \sum_{\tau \in \mathcal{T}} \left\| \mathrm{diag} \left( \textbf{U}^\top  \textbf{D}_{\tau}  \textbf{U}  \right) \right\|^2 \leq \sum_{\tau \in \mathcal{T}} \left\|  \textbf{U}^\top  \textbf{D}_{\tau}  \textbf{U} \right\|^2 = \sum_{\tau \in \mathcal{T}} \sum_{j = 1}^v \lambda_{\tau j}^2 p_j ,
	\end{align}
	with equality if and only if $  \textbf{U}^\top  \textbf{D}_{\tau}  \textbf{U} $, $ \tau \in \mathcal{T} $ are diagonal matrices, i.e. $ \textbf{U} $ is an eigenvector matrix of all $ \textbf{D}_{\tau} $, $ \tau \in \mathcal{T} $. One such matrix is  $ \textbf{U} = \textbf{I}_p $ and the maximal value of $ g(\textbf{U}) $ is thus indeed $ \sum_{\tau \in \mathcal{T}} \sum_{j = 1}^v \lambda_{\tau j}^2 p_j $.
	
	We next show that a both sufficient and necessary condition for $ \textbf{W} $ to be a maximizer of $ g $ is that $ \textbf{W} $ has, up to the ordering of its columns, the form 
	\begin{align}\label{eq:U_form} 
	\textbf{W} =
	\begin{pmatrix}
	\textbf{W}_{11} & \textbf{0} & \cdots & \textbf{0} \\
	\textbf{0} & \ddots & \ddots & \vdots \\
	\vdots & \ddots & \textbf{W}_{vv} & \textbf{0} \\
	\textbf{0} & \cdots & \textbf{0} & \textbf{W}_{00}
	\end{pmatrix} ,
	\end{align}
	where the partition into blocks is as in the statement of the lemma and the diagonal blocks $ \textbf{W}_{00}, \textbf{W}_{11}, \ldots , \textbf{W}_{vv}$ are orthogonal.
	
	We start with a the ``necessary''-part. Let $ \textbf{W} $ be an arbitrary maximizer of $ g $ and take its first column $ \textbf{w} = ( \textbf{w}_1^\top, \ldots , \textbf{w}_v^\top, \textbf{w}_0^\top)^\top $, partitioned in subvectors of lengths $ p_1, \ldots , p_v, p_0 $. As equality is reached in the inequality \eqref{eq:population_objective} only when $ \textbf{U} $ is an eigenvector of all $ \textbf{D}_\tau $, we have that $ \textbf{D}_\tau \textbf{w} = \pi_\tau \textbf{w}$ for some $ \pi_\tau \in \mathbb{R} $ for all $ \tau \in \mathcal{T} $. It then holds for all $ \tau $ that,
	\[
	\textbf{0} = (\textbf{D}_\tau - \pi_\tau \textbf{I}_p)\textbf{w} = 
	\begin{pmatrix}
	(\lambda_{\tau 1} - \pi_\tau) \textbf{I}_{p_1} & \textbf{0} & \cdots & \textbf{0}\\
	\textbf{0} & \ddots & \ddots & \vdots \\
	\vdots & \ddots & (\lambda_{\tau v} - \pi_\tau) \textbf{I}_{p_v} & \textbf{0} \\
	\textbf{0} & \cdots & \textbf{0} &  - \pi_\tau \textbf{I}_{p_0}
	\end{pmatrix}
	\begin{pmatrix}
	\textbf{w}_1 \\
	\vdots \\
	\textbf{w}_v \\
	\textbf{w}_0
	\end{pmatrix},
	\] 
	which yields the equation group,
	\[ 
	\textbf{0} =
	\begin{pmatrix}
	(\lambda_{\tau 1} - \pi_\tau) \textbf{w}_1 \\
	\vdots \\
	(\lambda_{\tau v} - \pi_\tau) \textbf{w}_v \\
	- \pi_\tau \textbf{w}_0
	\end{pmatrix}.
	\]
	We next proceed by proof through contradiction. Assume that two distinct subvectors of $ \textbf{w} $, say $ \textbf{w}_k$ and $ \textbf{w}_\ell $, both contain a non-zero element. Then
	\begin{align}\label{eq:eigen_consequence}
	\lambda_{\tau k} = \pi_\tau \quad \mbox{and} \quad \lambda_{\tau \ell} = \pi_\tau, \quad \forall \tau \in \mathcal{T},
	\end{align} 
	and we recall that $ \lambda_{\tau0} = 0$ for all $ \tau \in \mathcal{T} $. The identities \eqref{eq:eigen_consequence} imply that $ \lambda_{\tau k} = \lambda_{\tau \ell} $ for all $ \tau \in \mathcal{T} $, i.e., that the $ k $th and $ \ell $th blocks have perfectly matching autocovariance structures. If $ k \neq 0 $ and $ \ell \neq 0 $, this is a contradiction as the blocks were defined such that two distinct blocks always correspond to differing autocovariance structures. Moreover, if either $ k = 0 $ or $ \ell = 0 $, then $ \lambda_{\tau k} = \lambda_{\tau \ell} = 0 $ and we have found a signal (block) that has all autocovariances zero, contradicting Assumption \ref{assu:signal_from_noise}. Consequently, exactly one subvector of $ \textbf{w} $ is non-zero. As the choice of $ \textbf{w} $ within $ \textbf{W} $ was arbitrary, the result holds for all columns of $ \textbf{W} $. 
	
	We next show that exactly $ p_j $ columns of $ \textbf{W} $ have non-zero $ j $th subvector, $ j = 0, 1, \ldots , v $. Again the proof is by contradiction. Pick an arbitrary $ j = 0, 1, \ldots , v $ and assume that more than $ p_j $ columns of $ \textbf{W} $ are such that their non-zero part lies in the $ j $th subvector. Then these subvectors form a collection of more than $ p_j $ linearly independent vectors of length $ p_j $ (the linear independence follows as $ \textbf{W} $ is invertible and as each of its columns has non-zero elements in exactly one of the subvectors). This is a contradiction as no sets of linearly independent vectors with cardinality greater than $ n $ exist in $ \mathbb{R}^n $. Thus at most $ p_j $ columns of $ \textbf{W} $ have non-zero $ j $th subvector. Since the choice of $ j $ was arbitrary, the conclusion holds for all $ j = 0, 1, \ldots, v$ and we conclude that the size of the $ j $th block must be exactly $ p_j $. Ordering the columns now suitably shows that $ \textbf{W} $ must have the form \eqref{eq:U_form}, proving the first part of the argument.
	
	To see the sufficiency of the block diagonal form \eqref{eq:U_form} we first notice that any matrix $ \textbf{W} $ that can be column-permuted so that $ \textbf{W} \textbf{P} $ is of the form \eqref{eq:U_form} satisfies $ \textbf{D}_\tau \textbf{W} \textbf{P} = \textbf{W} \textbf{P} \textbf{D}_\tau $, $ \tau \in \mathcal{T} $. Thus,
	\begin{align*} 
	g(\textbf{W}) &= \sum_{\tau \in \mathcal{T}} \left\| \mathrm{diag} \left( \textbf{W}^\top  \textbf{D}_{\tau}  \textbf{W}  \right) \right\|^2\\
	&= \sum_{\tau \in \mathcal{T}} \left\| \mathrm{diag} \left( \textbf{P} \textbf{P}^\top \textbf{W}^\top  \textbf{D}_{\tau}  \textbf{W} \textbf{P} \textbf{P}^\top \right) \right\|^2\\
	&= \sum_{\tau \in \mathcal{T}} \left\|   \mathrm{diag} \left( \textbf{P} \textbf{D}_{\tau} \textbf{P}^\top \right) \right\|^2\\
	&= \sum_{\tau \in \mathcal{T}} \sum_{j = 1}^v \lambda_{\tau j}^2 p_j,
	\end{align*}   
	and we see that any $ \textbf{W} $ that is column-permutable to the form \eqref{eq:U_form} achieves the maximum. The sufficiency in conjunction with the necessity now equals the Fisher consistency of the population level problem.

	We next move to the sample properties of the sequence of SOBI-solutions $ \hat{\textbf{U}} $ and show the consistency of its ``off-diagonal blocks''. That is, we prove that any sequence of maximizers $ \hat{\textbf{U}} $ of $ \hat{g} $ can be permuted such that the off-diagonal blocks satisfy $ \hat{\textbf{U}}_{ij} \rightarrow_p \textbf{0} $. 
	
	Let the set of all $ p \times p $ orthogonal matrices be denoted by $ \mathcal{U}^{p} $. To temporarily get rid of the unidentifiability of the ordering of the columns, we work in a specific subset of $ \mathcal{U}^p $.
	\[
	\mathcal{U}_0 = \{ \textbf{U} = (\textbf{u}_1, \ldots, \textbf{u}_p) \in \mathcal{U}^p \mid \textbf{n}^\top \textbf{u}^2_1 \geq \cdots \geq  \textbf{n}^\top \textbf{u}_p^2 \},
	\]
	where $ \textbf{u}^2 \in \mathbb{R}^p$ is the vector of element-wise squares of $ \textbf{u} \in \mathbb{R}^p $ and $ \textbf{n} = (p, p-1, \ldots, 1)^\top $. All orthogonal matrices $\textbf{U} \in \mathcal{U}^p $ may have their columns permuted such that the permuted matrix belongs to $ \mathcal{U}_0 $. In case of ties in the condition defining $ \mathcal{U}_0 $, we arbitrarily choose one of the permutations. Let then $ \hat{\textbf{U}} $ be an arbitrary sequence of maximizers of $ \hat{g} $, every term of which we assume, without loss of generality, to be a member of $ \mathcal{U}_0 $.
	
	We first note that the uniform convergence of the sample objective function to the population objective function,
	\begin{align}\label{eq:uniform_convergence}
	\sup_{\textbf{U} \in \mathcal{U}_0} \left| \hat{g}(\textbf{U}) - g(\textbf{U}) \right| \rightarrow_p 0, 
	\end{align}
	can be seen to hold as in the proof of \cite[Theorem 1]{MiettinenIllnerNordhausenOjaTaskinenTheis2016}.
	
	Let the set of all $ \textbf{U} \in \mathcal{U}^p $ of the form \eqref{eq:U_form} be denoted by $ \mathcal{U}_P $ and define the set of all population level SOBI-solutions in $ \mathcal{U}_0 $ as
	\[ 
	\mathcal{U}_S = \{ \textbf{U} \in \mathcal{U}_0 \mid g(\textbf{U}) \geq g(\textbf{V}), \mbox{ for all } \textbf{V} \in \mathcal{U}_0  \}.
	\]
	We now claim that the set $ \mathcal{U}_0 $ is constructed such that we have $ \mathcal{U}_S \subset \mathcal{U}_P $. To see this, we prove the contrapositive claim that $ \mathcal{U} \setminus \mathcal{U}_{P} \subset \mathcal{U} \setminus \mathcal{U}_{S} $. Take an arbitrary $ \textbf{U} \in \mathcal{U} \setminus \mathcal{U}_{P} $. If $ \textbf{U} $ is not a maximizer of $ g $, then clearly $ \textbf{U} \in \mathcal{U} \setminus \mathcal{U}_{S} $ and we are done. If instead $ \textbf{U} $ is a maximizer of $ g $, then it must have two columns $ \textbf{u}_k, \textbf{u}_\ell $ such that $ k < \ell $ and $ \textbf{u}_k $ belongs to the $ i $th column block and $ \textbf{u}_\ell $ belongs to the $ j $th column block with $ i > j $ (the two columns are in wrong order with respect to $ \mathcal{U}_P $). However, then $ \textbf{n}^\top \textbf{u}^2_k \leq p - \sum_{k = 1}^{i-1} p_k < p - \sum_{k = 1}^j p_k + 1 \leq \textbf{n}^\top \textbf{u}^2_\ell $ and $ \textbf{U} \notin \mathcal{U}_0 $, implying that again $ \textbf{U} \in \mathcal{U} \setminus \mathcal{U}_{S} $. Us having exhausted all cases, any $ \textbf{U} \in \mathcal{U}_S $ is thus also a member of $ \mathcal{U}_P $ and has $ \textbf{U}_{ij} = 0 $ for all $ i \neq j $ where the partitioning is as in the statement of the lemma.
	
	We prove the consistency via showing that the sequence of solutions $ \hat{\textbf{U}} $ gets arbitrarily close to the solution set $ \mathcal{U}_S $ in the sense that,
	\[ 
	\mathbb{P}(\inf_{\textbf{V} \in \mathcal{U}_S} \| \hat{\textbf{U}} - \textbf{V} \|^2 > \varepsilon) \rightarrow 0, \quad \forall \varepsilon > 0.
	\]
	To see this, fix $ \varepsilon > 0 $ and define the $ \varepsilon $-neighbourhood of $ \mathcal{U}_S $ in $ \mathcal{U}_0 $ as
	\[ 
	\mathcal{U}_{S\varepsilon} = \{ \textbf{U} \in \mathcal{U}_0 \mid \inf_{\textbf{V} \in \mathcal{U}_S} \| \textbf{U} - \textbf{V} \|^2 \leq \varepsilon \}.
	\]
	Then
	\[ 
	\mathbb{P}(\inf_{\textbf{V} \in \mathcal{U}_S} \| \hat{\textbf{U}} - \textbf{V} \|^2 > \varepsilon) = \mathbb{P}(\hat{\textbf{U}} \in \mathcal{U}_0 \setminus \mathcal{U}_{S\varepsilon}). 
	\]
	As all maximizers of $ g $ in $ \mathcal{U}_0 $ lie in $ \mathcal{U}_{S} $, there exists $ \delta = \delta(\varepsilon) > 0 $ strictly positive such that for all $ \textbf{V} \in \mathcal{U}_0 \setminus \mathcal{U}_{S\varepsilon} $ we have $ g(\textbf{V}) < g(\textbf{U}_S) - \delta $ where $ \textbf{U}_S $ is an arbitrary element of $ \mathcal{U}_S $. This gives us,
	\[ 
	\mathbb{P}(\inf_{\textbf{V} \in \mathcal{U}_S} \| \hat{\textbf{U}} - \textbf{V} \|^2 > \varepsilon) \leq \mathbb{P}(  g(\textbf{U}_S) - g(\hat{\textbf{U}}) > \delta).
	\]
	By the definition of $ \hat{\textbf{U}} $ as a maximizer of $ \hat{g} $, we have $ \hat{g}(\hat{\textbf{U}}) \geq \hat{g}(\textbf{U}_S) $ and can construct the sequence of inequalities,
	\begin{align*}
	0 &\leq g(\textbf{U}_S) - g(\hat{\textbf{U}}) \\
	&\leq \hat{g}(\hat{\textbf{U}}) - g(\hat{\textbf{U}}) + g(\textbf{U}_S) - \hat{g}(\textbf{U}_S) \\
	&\leq 2 \sup_{\textbf{U} \in \mathcal{U}_0} \left| \hat{g}(\textbf{U}) - g(\textbf{U}) \right|,
	\end{align*}
	where invoking \eqref{eq:uniform_convergence} shows that $ g(\textbf{U}_S) - g(\hat{\textbf{U}}) \rightarrow_p 0 $. Consequently,
	\[ 
	\mathbb{P}(\inf_{\textbf{V} \in \mathcal{U}_S} \| \hat{\textbf{U}} - \textbf{V} \|^2 > \varepsilon) \leq \mathbb{P}(  g(\textbf{U}_S) - g(\hat{\textbf{U}}) > \delta) \rightarrow 0,
	\]
	and we have that $ \inf_{\textbf{V} \in \mathcal{U}_S} \| \hat{\textbf{U}} - \textbf{V} \|^2 = o_p(1) $. Writing this block-wise and remembering that all elements of $ \mathcal{U}_S \subset \mathcal{U}_P$ have off-diagonal blocks equal to zero, we get,
	\[ 
	\inf_{\textbf{V} \in \mathcal{U}_S} \| \hat{\textbf{U}} - \textbf{V} \|^2 = \inf_{\textbf{V} \in \mathcal{U}_S} \left( \sum_{i = 0}^v \| \hat{\textbf{U}}_{ii} - \textbf{V}_{ii} \|^2 + \sum_{i \neq j} \| \hat{\textbf{U}}_{ij} \|^2 \right) \geq \sum_{i \neq j} \| \hat{\textbf{U}}_{ij} \|^2,
	\]
	implying that all off-diagonal blocks of $ \hat{\textbf{U}} $ satisfy $\| \hat{\textbf{U}}_{ij} \| = o_p(1)$. Consequently, for every arbitrary sequence of solutions $ \hat{\textbf{U}} $, there exists a sequence of permutation matrices $ \hat{\textbf{P}} $ (chosen so that $  \hat{\textbf{U}} \hat{\textbf{P}} \in \mathcal{U}_0 $) such that the off-diagonal blocks of $ \hat{\textbf{U}} \hat{\textbf{P}} $ converge in probability to zero.

	\medskip
	
	\noindent{\small\textbf{2. Convergence rate}}
	
	\medskip
	
	We next establish that the off-diagonal blocks of any sequence of solutions $ \hat{\textbf{U}} \in \mathcal{U}_0 $ converge at the rate of root-$ T $. The claimed result then follows for an arbitrary sequence of solutions $ \hat{\textbf{U}} $ by choosing the sequence of permutations $ \hat{\textbf{P}} $ such that $  \hat{\textbf{U}} \hat{\textbf{P}} \in \mathcal{U}_0 $.
	
	By \cite[Definition 2]{MiettinenIllnerNordhausenOjaTaskinenTheis2016}, the estimating equations of the SOBI-solution $ \hat{\textbf{U}} = (\hat{\textbf{u}}_1, \ldots , \hat{\textbf{u}}_p)$ are,
	\begin{align}\label{eq:estimating_equations} 
	\textbf{u}_k^\top \sum_{\tau \in \mathcal{T}} \hat{\textbf{H}}_\tau \hat{\textbf{u}}_\ell  \textbf{u}_\ell^\top \hat{\textbf{H}}_\tau \hat{\textbf{u}}_\ell = \textbf{u}_\ell^\top \sum_{\tau \in \mathcal{T}} \hat{\textbf{H}}_\tau \hat{\textbf{u}}_k  \textbf{u}_k^\top \hat{\textbf{H}}_\tau \hat{\textbf{u}}_k, \quad \forall k, \ell = 1, \ldots , p,
	\end{align}
	along with the orthogonality constraint $ \textbf{U}^\top \textbf{U} = \textbf{I}_p $. The set of estimating equations \eqref{eq:estimating_equations} can be written in matrix form as,
	\[
	\sum_{\tau \in \mathcal{T}}  \hat{\textbf{U}}{}^\top \hat{\textbf{H}}_\tau \hat{\textbf{U}} \mathrm{diag}(\hat{\textbf{U}}{}^\top \hat{\textbf{H}}_\tau \hat{\textbf{U}}) =  \sum_{\tau \in \mathcal{T}} \mathrm{diag}(\hat{\textbf{U}}{}^\top \hat{\textbf{H}}_\tau \hat{\textbf{U}}) \hat{\textbf{U}}{}^\top \hat{\textbf{H}}_\tau \hat{\textbf{U}},
	\]
	which is equivalent to claiming that the matrix $ \hat{\textbf{Y}} = \sum_{\tau \in \mathcal{T}}  \hat{\textbf{U}}{}^\top \hat{\textbf{H}}_\tau \hat{\textbf{U}} \mathrm{diag}(\hat{\textbf{U}}{}^\top \hat{\textbf{H}}_\tau \hat{\textbf{U}})  $ is symmetric, $ \hat{\textbf{Y}} = \hat{\textbf{Y}}{}^\top $.
	
	We next take $ \hat{\textbf{Y}} $, multiply it by $ \sqrt{T} $ and expand as $\hat{\textbf{H}}_\tau = \hat{\textbf{H}}_\tau - \textbf{D}_\tau + \textbf{D}_\tau $ to obtain,
	\begin{align}\label{eq:y_expansion_H}
	\begin{split}
	\sqrt{T} \hat{\textbf{Y}} =& \sum_{\tau \in \mathcal{T}}  \hat{\textbf{U}}{}^\top \sqrt{T} (\hat{\textbf{H}}_\tau - \textbf{D}_\tau) \hat{\textbf{U}} \mathrm{diag}(\hat{\textbf{U}}{}^\top \hat{\textbf{H}}_\tau \hat{\textbf{U}}) \\
	+& \sum_{\tau \in \mathcal{T}}  \hat{\textbf{U}}{}^\top \textbf{D}_\tau \hat{\textbf{U}} \mathrm{diag}(\hat{\textbf{U}}{}^\top \sqrt{T} (\hat{\textbf{H}}_\tau - \textbf{D}_\tau) \hat{\textbf{U}}) \\
	+& \sqrt{T} \sum_{\tau \in \mathcal{T}}  \hat{\textbf{U}}{}^\top \textbf{D}_\tau \hat{\textbf{U}} \mathrm{diag}(\hat{\textbf{U}}{}^\top \textbf{D}_\tau \hat{\textbf{U}}).
	\end{split}
	\end{align}
	As $ \hat{\textbf{U}} = \mathcal{O}_p(1) $ by its orthogonality and $ \sqrt{T} (\hat{\textbf{H}}_\tau - \textbf{D}_\tau) = \mathcal{O}_p(1)$ by Lemma~\ref{lem:0}, the first two terms on the right-hand side of \eqref{eq:y_expansion_H} are bounded in probability and we may lump them under a single $ \mathcal{O}_p(1) $-symbol,
	\begin{align}\label{eq:y_form}
	\sqrt{T} \hat{\textbf{Y}} = \sqrt{T} \sum_{\tau \in \mathcal{T}}  \hat{\textbf{U}}{}^\top \textbf{D}_\tau \hat{\textbf{U}} \mathrm{diag}(\hat{\textbf{U}}{}^\top \textbf{D}_\tau \hat{\textbf{U}}) + \mathcal{O}_p(1).
	\end{align}
	Inspect next the term $ \hat{\textbf{D}}_\tau = \mathrm{diag} ( \hat{\textbf{U}}{}^\top \textbf{D}_\tau \hat{\textbf{U}} ) $. Performing the matrix multiplication block-wise we get as the $ (i, j) $th block of $\hat{\textbf{U}}{}^\top \textbf{D}_\tau \hat{\textbf{U}}$,
	\[ 
	( \hat{\textbf{U}}{}^\top \textbf{D}_\tau \hat{\textbf{U}} )_{ij} = \sum_{k = 0}^v \lambda_{\tau k} \hat{\textbf{U}}_{ki}^\top \hat{\textbf{U}}_{kj}.
	\]
	As $ \hat{\textbf{U}}{}_{ij}^\top \rightarrow_p \textbf{0}$ and  $ \hat{\textbf{U}}{}_{ii}^\top \hat{\textbf{U}}_{ii} \rightarrow_p \textbf{I}_{p_i}$ (the latter follows from the orthogonality of $ \hat{\textbf{U}} $ and the consistency of its off-diagonal blocks), we have,
	\[ 
	( \hat{\textbf{U}}{}^\top \textbf{D}_\tau \hat{\textbf{U}} )_{ij} = \delta_{ij} \lambda_{\tau i} \textbf{I}_{p_i} + o_p(1),
	\]
	where $ \delta_{\cdot \cdot} $ is the Kronecker delta. Consequently,
	\[ 
	\hat{\textbf{D}}_\tau = \mathrm{diag}( \hat{\textbf{U}}{}^\top \textbf{D}_\tau \hat{\textbf{U}}) = \textbf{D}_\tau + o_p(1).
	\]
	
	Denote by $ \hat{\textbf{U}}_{i, -j} \in \mathbb{R}^{(p - p_j) \times p_i} $ the $ i $th column block of $ \hat{\textbf{U}} $ with the $ j $th block removed, by $ \textbf{D}_{\tau, -j} \in \mathbb{R}^{(p - p_j) \times (p - p_j)} $ the result of removing the $ j $th column and row blocks of $ \textbf{D}_\tau $ and by $ \hat{\textbf{D}}_{\tau j} \rightarrow_p \lambda_{\tau j} \textbf{I}_{p_j} $ the $ j $th $ p_j \times p_j $ diagonal block of $ \hat{\textbf{D}}_\tau $. Our main claim is equivalent to requiring that,
	\[ 
	\hat{\textbf{U}}_{j, -j} = \mathcal{O}_p \left( \frac{1}{\sqrt{T}} \right), \quad \mbox{for all } j = 0, \ldots , v.
	\]
	To show this, fix next $ j $ and take the $ (i, j) $th block of the matrix $ \sqrt{T} \hat{\textbf{Y}} $ where $ i \neq j $ is arbitrary and separate the $ j $th term in the block-wise matrix multiplication of \eqref{eq:y_form} to obtain,
	\begin{align}\label{eq:separated_y} 
	\sqrt{T} \hat{\textbf{Y}}_{ij} = \sqrt{T} \sum_{\tau \in \mathcal{T}} \hat{\textbf{U}}_{i, -j}^\top \textbf{D}_{\tau, -j} \hat{\textbf{U}}_{j, -j} \hat{\textbf{D}}_{\tau j} + \sqrt{T} \sum_{\tau \in \mathcal{T}} \lambda_{\tau j} \hat{\textbf{U}}_{ji}^\top  \hat{\textbf{U}}_{jj} \hat{\textbf{D}}_{\tau j} + \mathcal{O}_p(1).
	\end{align}
	Opening up the $ (i, j) $th block (still with distinct $ i, j $) of the orthogonality constraint $ \hat{\textbf{U}}{}^\top \hat{\textbf{U}} = \textbf{I}_p$ and again separating the $ j $th term lets us write,
	\[ 
	\hat{\textbf{U}}_{ji}^\top  \hat{\textbf{U}}_{jj} = - \hat{\textbf{U}}_{i, -j}^\top \hat{\textbf{U}}_{j, -j}.
	\]
	Plugging this in to \eqref{eq:separated_y} gives us,
	\begin{align}\label{eq:separated_y_2} 
	\sqrt{T} \hat{\textbf{Y}}_{ij} = \sqrt{T} \sum_{\tau \in \mathcal{T}} \hat{\textbf{U}}_{i, -j}^\top \textbf{D}_{\tau, -j} \hat{\textbf{U}}_{j, -j} \hat{\textbf{D}}_{\tau j} - \sqrt{T} \sum_{\tau \in \mathcal{T}} \lambda_{\tau j} \hat{\textbf{U}}_{i, -j}^\top \hat{\textbf{U}}_{j, -j} \hat{\textbf{D}}_{\tau j} + \mathcal{O}_p(1).
	\end{align}
	Next we invoke the symmetry form, $ \sqrt{T} \hat{\textbf{Y}} = \sqrt{T} \hat{\textbf{Y}}{}^\top $, of the estimating equations \eqref{eq:estimating_equations}. In block form the equations claim that $\sqrt{T} \hat{\textbf{Y}}_{ij} = \sqrt{T} (\hat{\textbf{Y}}{}_{ji})^\top$. Performing now the expansion equivalent to \eqref{eq:separated_y_2} also for $ \sqrt{T} (\hat{\textbf{Y}}{}_{ji})^\top $ (again separating the $ j $th block in the summation) and plugging in the expansions into the symmetry relation, we obtain,
	\begin{align}\label{eq:a_step_1} 
	\begin{split}
	\mathcal{O}_p(1) =& \sqrt{T} \sum_{\tau \in \mathcal{T}} \hat{\textbf{U}}_{i, -j}^\top \textbf{D}_{\tau, -j} \hat{\textbf{U}}_{j, -j} \hat{\textbf{D}}_{\tau j} - \sqrt{T} \sum_{\tau \in \mathcal{T}} \lambda_{\tau j} \hat{\textbf{U}}_{i, -j}^\top \hat{\textbf{U}}_{j, -j} \hat{\textbf{D}}_{\tau j} \\
	-& \sqrt{T} \sum_{\tau \in \mathcal{T}} \hat{\textbf{D}}_{\tau i} \hat{\textbf{U}}_{i, -j}^\top \textbf{D}_{\tau, -j} \hat{\textbf{U}}_{j, -j} + \sqrt{T} \sum_{\tau \in \mathcal{T}} \lambda_{\tau j} \hat{\textbf{D}}_{\tau i} \hat{\textbf{U}}_{i, -j}^\top \hat{\textbf{U}}_{j, -j}.
	\end{split}
	\end{align}
	We then pre-multiply \eqref{eq:a_step_1} by $ \hat{\textbf{U}}_{i, -j} = \mathcal{O}_p(1) $ and sum the result over the index $ i \in \{0, \ldots , v\} \setminus \{ j \} $. Denoting $ \hat{\textbf{A}}_i = \hat{\textbf{U}}_{i, -j} $ this gives us,
	\begin{align}\label{eq:a_step_2} 
	\begin{split}
	\mathcal{O}_p(1) =& \sqrt{T} \sum_{\tau \in \mathcal{T}} \sum_{i \neq j} \hat{\textbf{A}}_i \hat{\textbf{A}}_i^\top \textbf{D}_{\tau, -j} \hat{\textbf{A}}_j \hat{\textbf{D}}_{\tau j} - \sqrt{T} \sum_{\tau \in \mathcal{T}} \sum_{i \neq j} \lambda_{\tau j} \hat{\textbf{A}}_i \hat{\textbf{A}}_i^\top \hat{\textbf{A}}_j \hat{\textbf{D}}_{\tau j} \\
	-& \sqrt{T} \sum_{\tau \in \mathcal{T}} \sum_{i \neq j} \hat{\textbf{A}}_i \hat{\textbf{D}}_{\tau i} \hat{\textbf{A}}_i^\top \textbf{D}_{\tau, -j} \hat{\textbf{A}}_j + \sqrt{T} \sum_{\tau \in \mathcal{T}} \sum_{i \neq j} \lambda_{\tau j} \hat{\textbf{A}}_i \hat{\textbf{D}}_{\tau i} \hat{\textbf{A}}_i^\top \hat{\textbf{A}}_j.
	\end{split}
	\end{align}
	We next row-vectorize \eqref{eq:a_step_2} to obtain us,
	\begin{align}\label{eq:a_step_3} 
	\begin{split}
	\mathcal{O}_p(1) =& \sum_{\tau \in \mathcal{T}} \sum_{i \neq j} \left[ \hat{\textbf{A}}_i \hat{\textbf{A}}_i^\top \textbf{D}_{\tau, -j} \otimes \hat{\textbf{D}}_{\tau j}  - \lambda_{\tau j} \hat{\textbf{A}}_i \hat{\textbf{A}}_i^\top \otimes \hat{\textbf{D}}_{\tau j} \right.\\
	&- \left. \hat{\textbf{A}}_i \hat{\textbf{D}}_{\tau i} \hat{\textbf{A}}_i^\top \textbf{D}_{\tau, -j} \otimes \textbf{I}_{p_j} + \lambda_{\tau j} \hat{\textbf{A}}_i \hat{\textbf{D}}_{\tau i} \hat{\textbf{A}}_i^\top \otimes \textbf{I}_{p_j} \right] \sqrt{T} \mathrm{vect}(\hat{\textbf{A}}_j).
	\end{split}
	\end{align}
	By the consistency of the off-diagonal blocks of $ \hat{\textbf{U}} $, we have $ \hat{\textbf{U}}_{ij} \rightarrow_p \textbf{0}$ for all $ i \neq j $ and $ \hat{\textbf{U}}_{ii} \hat{\textbf{U}}_{ii}^\top \rightarrow_p \textbf{I}_{p_i} $ for all $ i $. Consequently, we have the following convergences in probability, $ \sum_{i \neq j} \hat{\textbf{A}}_i \hat{\textbf{A}}{}_i^\top \rightarrow_p \textbf{I}_{p - p_j} $, $ \sum_{i \neq j} \hat{\textbf{A}}_i \hat{\textbf{D}}_{\tau i} \hat{\textbf{A}}{}_i^\top \rightarrow_p \textbf{D}_{\tau, -j} $ and $ \hat{\textbf{D}}_{\tau j} \rightarrow_p \lambda_{\tau j} \textbf{I}_{p_j} $. Calling next the matrix in the square brackets on the right-hand side of \eqref{eq:a_step_3} by $ \hat{\textbf{C}} \in \mathbb{R}^{(p - p_j)p_j \times (p - p_j)p_j}$, the convergences imply that,
	\begin{align}\label{eq:coefficient_matrix} 
	\hat{\textbf{C}} \rightarrow_p \textbf{C} = \sum_{\tau \in \mathcal{T}} \left[ \lambda_{\tau j} \textbf{D}_{\tau, -j} \otimes \textbf{I}_{p_j} - \lambda_{\tau j}^2 \textbf{I}_{(p - p_j)p_j} - \textbf{D}_{\tau, -j}^2 \otimes \textbf{I}_{p_j} + \lambda_{\tau j} \textbf{D}_{\tau, -j} \otimes \textbf{I}_{p_j} \right].
	\end{align}
	The matrix $ \textbf{C} $ in\eqref{eq:coefficient_matrix} is a diagonal matrix and its diagonal is divided into $ v $ segments of lengths $ p_i p_j $, $ i \in \{0, \ldots , v\} \setminus \{ j \} $. Each segment matches with the vectorization of the corresponding block $ \hat{\textbf{U}}_{ij} $ in the vectorized matrix $\mathrm{vect}( \hat{\textbf{A}}_j ) = \mathrm{vect}( \hat{\textbf{U}}_{j, -j} )$. All elements of the $ i $th segment of the diagonal of $ \textbf{C} $ are equal to,
	\[ 
	\sum_{\tau \in \mathcal{T}} \left( \lambda_{\tau j} \lambda_{\tau i} - \lambda_{\tau j}^2 - \lambda_{\tau i}^2 +  \lambda_{\tau j} \lambda_{\tau i} \right) = -\sum_{\tau \in \mathcal{T}} \left(  \lambda_{\tau i} - \lambda_{\tau j} \right)^2 < 0,
	\]  
	where the inequality follows from our definition of the blocks such that they differ in their autocovariances for at least one lag $ \tau \in \mathcal{T} $. Thus the matrix $ \textbf{C} $ is invertible and we may pre-multiply \eqref{eq:a_step_3} by $ \hat{\textbf{C}}{}^{-1} $ which is asymptotically well-defined. By Slutsky's theorem (for random matrices) we obtain,
	\begin{align}\label{eq:boundedness} 
	\sqrt{T} \mathrm{vect}(\hat{\textbf{A}}_j) = \hat{\textbf{C}}^{-1} \mathcal{O}_p(1) = \mathcal{O}_p(1).
	\end{align}
	As the choice of the column block $ j $ was arbitrary, the result \eqref{eq:boundedness} holds for all $ \hat{\textbf{A}}_j = \hat{\textbf{U}}_{j, -j} $, concluding the proof of Lemma \ref{lem:1}.
\end{proof}

\begin{proof}[Proof of Corollary \ref{cor:1}]
	The $ j $th diagonal block of the orthogonality constraint $ \hat{\textbf{U}}{}^\top \hat{\textbf{U}} = \textbf{I}_p $ reads,
	\[ 
	\sum_{k \neq j} \hat{\textbf{U}}_{kj}^\top \hat{\textbf{U}}_{kj} = \textbf{I}_{p_j} -  \hat{\textbf{U}}_{jj}^\top \hat{\textbf{U}}_{jj},
	\]
	where the left-hand side is by Lemma \ref{lem:1} of order $ \mathcal{O}_p(1/T) $, giving the first claim. The second one follows in a similar way by starting with $ \hat{\textbf{U}} \hat{\textbf{U}}{}^\top = \textbf{I}_p $ instead.
\end{proof}

\begin{proof}[Proof of Lemma \ref{lem:2}]

	Recall the definition of $ \hat{m}_q $ as,
	\[ 
	\hat{m}_q = \frac{1}{|\mathcal{T}| r^2} \sum_{\tau \in \mathcal{T}} \| \hat{\textbf{W}}_q^\top \hat{\textbf{H}}_\tau \hat{\textbf{W}}_q \|^2,
	\]
	where $ \hat{\textbf{W}}_q $ contains the columns of the SOBI-solution that correspond to the smallest $ q $ sums of squared pseudo-eigenvalues $ \sum_{\tau \in \mathcal{T}} \mathrm{diag}(\hat{\textbf{U}}{}^\top \hat{\textbf{H}}_\tau \hat{\textbf{U}})^2 $.
	
	By Lemma \ref{lem:1}, $ \hat{\textbf{U}}{}^\top \hat{\textbf{H}}_\tau \hat{\textbf{U}} = \hat{\textbf{P}} \tilde{\textbf{U}}{}^\top \hat{\textbf{H}}_\tau \tilde{\textbf{U}} \hat{\textbf{P}}{}^\top$ where $ \tilde{\textbf{U}} $ is the block-diagonal matrix on the right-hand side of Lemma \ref{lem:1}. We derive an asymptotic expression for the $ i $th diagonal block $ \hat{\textbf{E}}_{\tau ii} $ of the matrices $ \hat{\textbf{E}}_\tau = \tilde{\textbf{U}}{}^\top \hat{\textbf{H}}_\tau \tilde{\textbf{U}} $. By Lemmas~\ref{lem:0}, \ref{lem:1}, Corollary \ref{cor:1} and Assumption \ref{assu:consistent}, 
	\begin{align}\label{eq:pseudo_expansion}
	\begin{split}
	\hat{\textbf{E}}_{\tau ii} &= \sum_{s = 0}^v \sum_{t = 0}^v \hat{\textbf{U}}{}^\top_{si} \hat{\textbf{H}}_{\tau st} \hat{\textbf{U}}{}_{ti} \\
	&= \sum_{s \neq t} \hat{\textbf{U}}{}^\top_{si} \hat{\textbf{H}}_{\tau st} \hat{\textbf{U}}{}_{ti} + \sum_{s = 0}^v \hat{\textbf{U}}{}^\top_{si} ( \hat{\textbf{H}}_{\tau ss} - \lambda_{\tau s} \textbf{I}_{p_s}) \hat{\textbf{U}}{}_{si} + \sum_{s = 0}^v \lambda_{\tau s} \hat{\textbf{U}}{}^\top_{si} \hat{\textbf{U}}{}_{si} \\
	&= \lambda_{\tau i} \textbf{I}_{p_s} + \hat{\textbf{U}}{}^\top_{ii} ( \hat{\textbf{H}}_{\tau ii} - \lambda_{\tau i} \textbf{I}_{p_i}) \hat{\textbf{U}}{}_{ii} + \mathcal{O}_p(1/T),
	\end{split}
	\end{align}
	where $ \hat{\textbf{H}}_{\tau st} $ is the $ (s, t) $th block of $ \hat{\textbf{H}}_{\tau} $ in the indexing of Lemma \ref{lem:1}. As $ ( \hat{\textbf{H}}_{\tau ii} - \lambda_{\tau i} \textbf{I}_{p_i} ) = \mathcal{O}_p(1/\sqrt{T}) $, we have by \eqref{eq:pseudo_expansion} that the pseudo-eigenvalues converge in probability to the respective population values,
	\begin{align}\label{eq:pseudo_convergence}
	 \sum_{\tau \in \mathcal{T}} \mathrm{diag}(\hat{\textbf{E}}_\tau)^2 \rightarrow_p \sum_{\tau \in \mathcal{T}} \boldsymbol{\Lambda}^2_\tau.
	\end{align}
	Let $ A_q $ denote the event that the last $ q $ columns of $ \hat{\textbf{U}} $ are up to ordering equal to the last $ q $ columns of $ \tilde{\textbf{U}} $, that is, the ordering based on the estimated sums of squared pseudo-eigenvalues correctly identifies the noise components. By Assumption \ref{assu:signal_from_noise}, the signals are well-separated from the noise in the sense that no signal corresponds to the value zero in the diagonal of $ \sum_{\tau \in \mathcal{T}} \boldsymbol{\Lambda}^2_\tau $ and consequently, by \eqref{eq:pseudo_convergence}, we have $ \mathbb{P}(A_q) \rightarrow 1 $.
	
	Denote next the final column block of $ \tilde{\textbf{U}} $ by $ \tilde{\textbf{U}}_q \in \mathbb{R}^{r} $. Conditional on $ A_q $, the two column blocks are the same up to a permutation, $  \hat{\textbf{W}}_q = \tilde{\textbf{U}}_q \hat{\textbf{P}}_q  $ for some sequence of permutation matrices $ \hat{\textbf{P}}_q \in \mathbb{R}^{r \times r}$, and we can write for an arbitrary $ \varepsilon > 0$,
	\begin{align*}
	& \mathbb{P}\left(\sqrt{T} \left| \| \hat{\textbf{W}}_q^\top \hat{\textbf{H}}_\tau \hat{\textbf{W}}_q \| - \| \tilde{\textbf{U}}_q^\top \hat{\textbf{H}}_\tau \tilde{\textbf{U}}_q \| \right| < \epsilon \right) \\
	=& \mathbb{P}\left(\sqrt{T} \left| \| \hat{\textbf{W}}_q^\top \hat{\textbf{H}}_\tau \hat{\textbf{W}}_q \| - \| \tilde{\textbf{U}}_q^\top \hat{\textbf{H}}_\tau \tilde{\textbf{U}}_q \| \right|  < \epsilon \mid A_q \right) \mathbb{P}(A_q)\\
	+& \mathbb{P}\left(\sqrt{T} \left| \| \hat{\textbf{W}}_q^\top \hat{\textbf{H}}_\tau \hat{\textbf{W}}_q \| - \| \tilde{\textbf{U}}_q^\top \hat{\textbf{H}}_\tau \tilde{\textbf{U}}_q \| \right|  < \epsilon \mid A_q^c \right) \mathbb{P}(A_q^c) \\
	=&  \mathbb{P}(A_q) + \mathbb{P}\left(\sqrt{T} \left| \| \hat{\textbf{W}}_q^\top \hat{\textbf{H}}_\tau \hat{\textbf{W}}_q \| - \| \tilde{\textbf{U}}_q^\top \hat{\textbf{H}}_\tau \tilde{\textbf{U}}_q \| \right|  < \epsilon \mid A_q^c \right) ( 1 - \mathbb{P}(A_q) ) \rightarrow 1,
	\end{align*}
	showing the convergence in probability, $ \sqrt{T} \| \hat{\textbf{W}}_q^\top \hat{\textbf{H}}_\tau \hat{\textbf{W}}_q \| = \sqrt{T} \| \tilde{\textbf{U}}_q^\top \hat{\textbf{H}}_\tau \tilde{\textbf{U}}_q \| + o_p(1) $, for all $ \tau \in \mathcal{T} $. Furthermore, by \eqref{eq:pseudo_expansion},
	\[
	\sqrt{T} \| \tilde{\textbf{U}}_q^\top \hat{\textbf{H}}_\tau \tilde{\textbf{U}}_q \| = \sqrt{T} \| \hat{\textbf{E}}_{\tau 00} \| = \| \sqrt{T} \hat{\textbf{U}}{}^\top_{00}  \hat{\textbf{H}}_{\tau 00} \hat{\textbf{U}}{}_{00} + \mathcal{O}_p(1/\sqrt{T}) \| = \mathcal{O}_p(1),
	\]
	showing that,
	\begin{align*}
	T \| \hat{\textbf{W}}_q^\top \hat{\textbf{H}}_\tau \hat{\textbf{W}}_q \|^2 &= \| \sqrt{T} \hat{\textbf{U}}{}^\top_{00}  \hat{\textbf{H}}_{\tau 00} \hat{\textbf{U}}{}_{00} + \mathcal{O}_p(1/\sqrt{T}) \|^2 + o_p(1) \\
	&= \| \sqrt{T} \hat{\textbf{U}}{}^\top_{00}  \hat{\textbf{H}}_{\tau 00} \hat{\textbf{U}}{}_{00} \|^2 + o_p(1) \\
	&= T \cdot \mathrm{tr}(\hat{\textbf{U}}{}_{00} \hat{\textbf{U}}{}^\top_{00} \hat{\textbf{H}}_{\tau 00} \hat{\textbf{U}}{}_{00} \hat{\textbf{U}}{}^\top_{00} \hat{\textbf{H}}_{\tau 00} ) + o_p(1) \\
	&= T \| \hat{\textbf{H}}_{\tau 00} \|^2 + o_p(1) \\
	&= T \| \hat{\textbf{R}}_{\tau 00} \|^2 + o_p(1),
	\end{align*}
	where the second-to-last equality uses Corollary~\ref{cor:1} and the last one Lemma~\ref{lem:0}. Substituting now into the definition of $ \hat{m}_q $, we obtain the claim,
	\[ 
	T \cdot \hat{m}_q = \frac{T}{|\mathcal{T}| r^2} \sum_{\tau \in \mathcal{T}} \| \hat{\textbf{W}}_q^\top \hat{\textbf{H}}_\tau \hat{\textbf{W}}_q \|^2 = \frac{T}{|\mathcal{T}| r^2} \sum_{\tau \in \mathcal{T}} \| \hat{\textbf{R}}_{\tau 00} \|^2 + o_p(1).
	\]
\end{proof}

\begin{proof}[Proof of Lemma \ref{lem:3}]
	
	Write first,
	\begin{align*} 
	\hat{\textbf{S}}_\tau &= \frac{1}{T - \tau} \sum_{t=1}^{T - \tau} (\textbf{x}_t  - \bar{\textbf{x}}) (\textbf{x}_{t + \tau}  - \bar{\textbf{x}})^\top \\
	&=  \frac{1}{T - \tau} \sum_{t=1}^{T - \tau} \textbf{x}_t  \textbf{x}_{t + \tau}^\top -  \bar{\textbf{x}} \frac{1}{T - \tau} \sum_{t=1}^{T - \tau} \textbf{x}_{t + \tau}^\top -  \frac{1}{T - \tau} \sum_{t=1}^{T - \tau} \textbf{x}_t \bar{\textbf{x}}^\top + \bar{\textbf{x}} \bar{\textbf{x}}^\top.
	\end{align*}
	By Assumption \ref{assu:ma_infinity} and \cite[Proposition 11.2.2]{BrockwellDavis1991}, the latent series $ \textbf{z}_t = \textbf{x}_t $ (we use identity mixing) satisfy a central limit theorem, implying that $ \bar{\textbf{x}} = \mathcal{O}_p(1/\sqrt{T})$. Thus,
	\[ 
	\sqrt{T} ( \hat{\textbf{S}}_\tau - \textbf{D}_\tau ) = \sqrt{T} ( \frac{1}{T - \tau} \sum_{t=1}^{T - \tau} \textbf{x}_t  \textbf{x}_{t + \tau}^\top  - \textbf{D}_\tau ) + \mathcal{O}_p(1/\sqrt{T}),
	\]
	and it is sufficient to show the limiting result for the non-centered covariance and autocovariance matrices. Consequently, in the following we implicitly assume that no centering is used.
	
	The blocks $  \hat{\textbf{R}}_{\tau_100}, \ldots , \hat{\textbf{R}}_{\tau_{|\mathcal{T}|}00} $ are the symmetrized autocovariance matrices of the white noise part of $ \textbf{z}_t $. By Assumption \ref{assu:ma_infinity}, the latent series $ \textbf{z}_t $ has an $ \mathrm{MA}(\infty) $-representation and by considering only the last $ r $ components of the representation we see that also the white noise part has separately an $ \mathrm{MA}(\infty) $-representation. Now the lower right diagonal blocks of the matrices $ \boldsymbol{\Psi}_j $ take the roles of $ \boldsymbol{\Psi}_j $ and by Assumption \ref{assu:ma_infinity} these blocks equal $ \boldsymbol{\Psi}_{j00} = \delta_{j0} \textbf{I}_r $. Consequently, by \cite[Lemma 1]{MiettinenIllnerNordhausenOjaTaskinenTheis2016} the vector,
	\[ 
	\sqrt{T} \mathrm{vec} \left( \hat{\textbf{R}}_{\tau_100}, \ldots , \hat{\textbf{R}}_{\tau_{|\mathcal{T}|}00} \right),
	\]
	admits a limiting multivariate normal distribution with zero mean and the covariance matrix equal to
	\begin{align}\label{eq:covariance_matrix_V} 
	\textbf{V} = \begin{pmatrix}
	\textbf{V}_{11} & \cdots & \textbf{V}_{1|\mathcal{T}|} \\
	\vdots & \ddots & \vdots \\
	\textbf{V}_{|\mathcal{T}|1} & \cdots & \textbf{V}_{|\mathcal{T}||\mathcal{T}|}
	\end{pmatrix} \in \mathbb{R}^{|\mathcal{T}| r^2 \times |\mathcal{T}| r^2},
	\end{align}
	where $ \textbf{V}_{\ell m} = \mathrm{diag}(\mathrm{vec}(\textbf{D}_{\ell m})) (\textbf{K}_{rr} - \textbf{D}_{rr} + \textbf{I}_{r^2} )$. The matrices $ \textbf{D}_{\ell m} \in \mathbb{R}^{r \times r} $, $ \ell, m = 0, \ldots , |\mathcal{T}| $ (we do not use the zero index here but it appears in the following formulas so it is included), are defined element-wise as,
	\begin{align}\label{eq:D_matrices}
	\begin{split} 
	(\textbf{D}_{\ell m})_{ii} &= (\beta_i - 3) (\textbf{F}_\ell)_{ii} (\textbf{F}_m)_{ii} + \sum_{k = -\infty}^\infty \left[  (\textbf{F}_{k + \ell})_{ii} (\textbf{F}_{k + m})_{ii} + (\textbf{F}_{k + \ell})_{ii} (\textbf{F}_{k - m})_{ii}   \right] \\
	(\textbf{D}_{\ell m})_{ij} &= (\beta_{ij} - 1) (\textbf{F}_\ell + \textbf{F}_\ell^\top)_{ij} (\textbf{F}_m + \textbf{F}_m^\top)_{ij} \\
	&+ \frac{1}{2} \sum_{k = -\infty}^\infty \left[  (\textbf{F}_{k + \ell - m})_{ii} (\textbf{F}_{k})_{jj} + (\textbf{F}_{k})_{ii} (\textbf{F}_{k + \ell + m})_{jj} \right], \quad i \neq j.
	\end{split}
	\end{align}
	where $ \beta_i = \mathrm{E}(\epsilon_{ti}^4) $, $ \beta_{ij} = \mathrm{E}(\epsilon_{ti}^2 \epsilon_{ti}^2) $ and $ \epsilon_{ti} $, $ i = 1, \ldots ,r $, refers to the $ i $th innovation component in the $ \mathrm{MA}(\infty) $-representation of the white noise part. The matrices $ \textbf{F}_\ell $ are defined as $ \textbf{F}_\ell = \sum_{j=-\infty}^\infty \boldsymbol{\psi}_j \boldsymbol{\psi}_{j + \ell}^\top $ where the vectors $ \boldsymbol{\psi}_j \in \mathbb{R}^r$ contain the diagonal elements of the matrices $ \boldsymbol{\Psi}_{j00} $.
	
	Under Assumption \ref{assu:ma_infinity} we have $ \boldsymbol{\psi}_j = \delta_{j0} \boldsymbol{1}_r $ where the vector $ \boldsymbol{1}_r \in \mathbb{R}^r$ consists solely of zeroes. Consequently $ \textbf{F}_\ell = \delta_{\ell 0} \textbf{J}_r $. Plugging this in to \eqref{eq:D_matrices} gives for the diagonal elements of $ \textbf{D}_{\ell m} $ that
	\begin{align*}
	(\textbf{D}_{\ell m})_{ii} &= (\beta_i - 3) \delta_{\ell 0} \delta_{m 0} + \sum_{k = -\infty}^\infty \left[ \delta_{(k + \ell) 0} \delta_{(k + m) 0} + \delta_{(k + \ell) 0}\delta_{(k - m) 0} \right] \\
	&= (\beta_i - 3) \delta_{\ell 0} \delta_{m 0} + \delta_{\ell m} + \delta_{\ell 0} \delta_{m 0},
	\end{align*}
	which implies that the matrices $ \textbf{V}_{\ell m} $, $ \ell, m = 1, \ldots , |\mathcal{T}| $, have non-zero diagonals precisely when $ \ell = m $ and then the diagonal is filled with ones. Plugging $ \textbf{F}_\ell = \delta_{\ell 0} \textbf{J}_r $ in to the definition of the diagonal elements in \eqref{eq:D_matrices} gives,
	\begin{align*}
	(\textbf{D}_{\ell m})_{ij} &= (\beta_{ij} - 1) 2 \delta_{\ell 0} 2 \delta_{\ell 0} + \frac{1}{2} \sum_{k = -\infty}^\infty \left[  \delta_{(k + \ell - m)0} \delta_{k0} + \delta_{k0} \delta_{(k + \ell + m)0} \right] \\
	&= 4 (\beta_{ij} - 1) \delta_{\ell 0} \delta_{\ell 0} + \frac{1}{2} (\delta_{\ell m} + \delta_{\ell 0} \delta_{m 0}),
	\end{align*}
	which says that the matrices $ \textbf{V}_{\ell m} $, $ \ell, m = 1, \ldots , |\mathcal{T}| $, have non-zero off-diagonals precisely when $ \ell = m $ and then the off-diagonal is filled with one-halves.
	
	Combining the forms for the diagonals and off-diagonals, we get $ \textbf{D}_{\ell m} = \delta_{\ell m} ( \textbf{J}_r + \textbf{I}_r)/2$, $ \ell, m = 1, \ldots , | \mathcal{T} | $. Plugging this in to \eqref{eq:covariance_matrix_V} now gives the claim and concludes the proof.
	
\end{proof}

\begin{proof}[Proof of Proposition \ref{prop:1}]
	By Lemma \ref{lem:2}, the limiting distribution of $ T | \mathcal{T} | r^2 \cdot \hat{m}_q $ is the same as the limiting distribution of
	\begin{align*} 
	T | \mathcal{T} | r^2 \cdot \hat{m}_q^* 
	&= \sum_{\tau \in \mathcal{T}} \| \sqrt{T} \hat{\textbf{R}}_{\tau00} \|^2\\
	&= \| \sqrt{T} \mathrm{vec} \left( \hat{\textbf{R}}_{\tau_100}, \ldots , \hat{\textbf{R}}_{\tau_{|\mathcal{T}|}00} \right) \|^2\\
	&= \sqrt{T} \mathrm{vec}^\top \left( \hat{\textbf{R}}_{\tau_100}, \ldots , \hat{\textbf{R}}_{\tau_{|\mathcal{T}|}00} \right) \sqrt{T} \mathrm{vec} \left( \hat{\textbf{R}}_{\tau_100}, \ldots , \hat{\textbf{R}}_{\tau_{|\mathcal{T}|}00} \right).
	\end{align*}
	By Lemma~\ref{lem:3} and the continuous mapping theorem, the limiting distribution of $ T | \mathcal{T} | r^2 \cdot \hat{m}_q^* $ is the same as the distribution of $ \textbf{y}^\top \textbf{y} $ where $ \textbf{y} $ is a mean-zero multivariate normal random vector with the covariance matrix $ \textbf{V} $ given in Lemma~\ref{lem:3}. Equivalently, the limiting distribution of $ T | \mathcal{T} | r^2 \cdot \hat{m}_q^* $ is the same as the distribution of $ \textbf{y}_0^\top \textbf{V} \textbf{y}_0 $ where $ \textbf{y}_0 $ is a standardized multivariate normal random vector. By \cite[Chapter 3.5]{serfling2009approximation}, if $ \textbf{V} $ is idempotent and symmetric, then the limiting distribution of $ \textbf{y}_0^\top \textbf{V} \textbf{y}_0 $ is $ \chi^2_{\mathrm{tr}(\textbf{V})} $. To see that $ \textbf{V} $ is indeed idempotent, we inspect the square of its arbitrary diagonal block $ \textbf{V}_0 $,
	\begin{align*}
	\textbf{V}_0^2 = \left[ \mathrm{diag}(\mathrm{vec}(\textbf{J}_{r} + \textbf{I}_{r})/2) (\textbf{K}_{rr} - \textbf{D}_{rr} + \textbf{I}_{r^2} ) \right]^2.
	\end{align*}
	We simplify using $ \mathrm{diag}(\mathrm{vec}(\textbf{J}_{r})) = \textbf{I}_{r^2} $,  $\mathrm{diag}(\mathrm{vec}(\textbf{I}_{r})) = \textbf{D}_{rr} $, $ \textbf{D}_{rr} \textbf{K}_{rr} = \textbf{D}_{rr}$, $ \textbf{D}_{rr}^2 = \textbf{D}_{rr} $ and $ \textbf{K}_{rr}^2 = \textbf{I}_{r^2} $, to obtain $ \textbf{V}_0 = (\textbf{K}_{rr} + \textbf{I}_{r^2})/2 $, which is symmetric, and,
	\[ 
	\textbf{V}_0^2 = \left[ \frac{1}{2} (\textbf{K}_{rr} + \textbf{I}_{r^2}) \right]^2 = \frac{1}{4} (2 \textbf{K}_{rr} + 2 \textbf{I}_{r^2}) = \textbf{V}_0.
	\]
	Thus $ \textbf{V}_0 $ is idempotent and symmetric and consequently $ \textbf{V} $, constituting solely of the $ |\mathcal{T}| $ diagonal blocks each equal to $ \textbf{V}_0 $, is also idempotent and symmetric. The trace of $ \textbf{V} $ is $ |\mathcal{T}| $ times the trace of $ \textbf{V}_0 $, which equals,
	\[ 
	\mathrm{tr}(\textbf{V}_0) = \frac{1}{2} \mathrm{tr}(\textbf{K}_{rr}) + \frac{1}{2} \mathrm{tr}(\textbf{I}_{r^2}) = \frac{1}{2}(r + r^2) = \frac{1}{2}r(r + 1).
	\]
	The trace of $ \textbf{V} $ is then $ |\mathcal{T}| r(r + 1)/2 $ and we have proved that the limiting distribution of  $ \textbf{y}_0^\top \textbf{V} \textbf{y}_0 $, and consequently that of $  T | \mathcal{T} | r^2 \cdot \hat{m}_q $, is  $ \chi^2_{| \mathcal{T} |r(r + 1)/2} $.
\end{proof}

\markboth{\hfill{\footnotesize\rm JONI VIRTA AND KLAUS NORDHAUSEN} \hfill}
{\hfill {\footnotesize\rm DETERMINING THE DIMENSION IN SECOND ORDER SEPARATION} \hfill}


\bibliography{refs}
\bibliographystyle{Chicago}

\vskip .65cm
\noindent
Aalto University, Finland
\vskip 2pt
\noindent
E-mail: joni.virta@aalto.fi
\vskip 2pt

\noindent
Vienna University of Technology
\vskip 2pt
\noindent
E-mail: klaus.nordhausen@tuwien.ac.at
\end{document}